\documentclass[reqno]{amsart}

\usepackage{amsmath}
\usepackage{amsfonts}
\usepackage{amssymb,enumerate}
\usepackage{amsthm}
\usepackage[all]{xy}
\usepackage{rotating}
\usepackage{hyperref}

\theoremstyle{plain}
\newtheorem{lem}{Lemma}[section]
\newtheorem{cor}[lem]{Corollary}
\newtheorem{prop}[lem]{Proposition}
\newtheorem{thm}[lem]{Theorem}

\theoremstyle{definition}
\newtheorem{defn}[lem]{Definition}
\newtheorem{ex}[lem]{Example}
\newtheorem{question}[lem]{Question}

\newtheorem{disc}[lem]{Remark}

\newtheorem{notn}[lem]{Notation}

\newcommand{\im}{\operatorname{Im}}

\newcommand{\Cl}{\operatorname{Cl}}

\newcommand{\ideal}[1]{\mathfrak{#1}}

\newcommand{\p}{\ideal{p}}
\newcommand{\q}{\ideal{q}}

\newcommand{\fa}{\ideal{a}}
\newcommand{\fb}{\ideal{b}}
\newcommand{\fc}{\ideal{c}}

\newcommand{\fs}{\ideal{s}}
\newcommand{\fr}{\ideal{r}}

\newcommand{\bbz}{\mathbb{Z}}
\newcommand{\bbn}{\mathbb{N}}
\newcommand{\bbq}{\mathbb{Q}}
\newcommand{\bbr}{\mathbb{R}}
\newcommand{\bbc}{\mathbb{C}}

\newcommand{\xra}{\xrightarrow}

\newcommand{\into}{\hookrightarrow}

\renewcommand{\geq}{\geqslant}
\renewcommand{\leq}{\leqslant}

\newcommand{\ssm}{\smallsetminus}

\newcommand{\ft}{\ideal{t}}

\theoremstyle{definition}
\newtheorem*{assumptions}{Assumptions}

\numberwithin{equation}{lem}

\begin{document}

\bibliographystyle{amsplain}


\author{Sean K. Sather-Wagstaff}

\address{School of Mathematical and Statistical Sciences,
Clemson University,
O-110 Martin Hall, Box 340975, Clemson, S.C. 29634
USA}

\email
{ssather@clemson.edu}

\urladdr
{https://ssather.people.clemson.edu/}

\title{Irreducible divisor pair domains}




\begin{abstract}
We introduce and study a new class of integral domains which we call irreducible divisor pair domains (IDPDs). In particular, we show how IDPDs fit in with other classes of integral domains defined in terms of factorization conditions. For instance, every UFD is an IDPD, and every IDPD is an HFD, but the converses fail in general. We also show that many familiar examples of HFDs are also IDPDs.
\end{abstract}

\maketitle


\section{Introduction} \label{sec170725a}

\begin{assumptions}
Throughout this paper, let $D$ be an integral domain with unit group $U(D)$.
Let $\sim$ denote the associate relation on $D$.
Recall that the integral domain $D$ is \emph{atomic}
if every non-zero non-unit $z$ of $D$ factors as a finite product $z=p_1\ldots p_n$ of irreducible elements
or ``atoms'', hence the terminology.
\end{assumptions}

The goal of this paper is to initiate the study of the  class irreducible divisor pair domains, defined next.

\begin{defn}\label{defn170725a}
The integral domain
$D$ is a \emph{irreducible divisor pair domain (IDPD)}
if it is atomic
such that for every non-zero non-unit $z\in D$ and for every pair $p,q\in D$ of non-associate irreducible divisors
of $z$, there exist atoms $p',q'\in D$ such that $pp'\sim qq'\mid z$.
\end{defn}

This definition is motivated by the work of Coykendall and Maney~\cite{coykendall:idg}
and our own work with Goodell~\cite{goodell:cf} on combinatorial tools for understanding
factorization in integral domains.
See Remark~\ref{disc170727a} for a discussion of the connection with the first of these.

As part of this initial investigation of IDPDs, we are interested in where these domains fit in with other
classes of domains. For instance:

\begin{ex}\label{ex170725a}
If $D$ is a unique factorization domain (UFD), then it is an IDPD.
Indeed, assume that $D$ is a UFD, so $D$ is atomic and every atom of $D$ is prime. 
Consider a non-zero non-unit $z\in D$ and a pair $p,q\in D$ of non-associate irreducible divisors
of $z$. Write $z=pz'$ for some $z'\in D\ssm\{0\}$, so we have $q\mid z=pz'$.
Since $q$ is prime and is not associate to the prime $p$, it follows that $q\mid z'$.
From this we have $pq\mid pz'=z$, and it follows that the defining condition from~\ref{defn170725a} is satisfied with $p'=q$ and $q'=p$.
(One can also deduce this implication as an application of~\cite[Theorem~5.1]{coykendall:idg}.)
\end{ex}

We are also interested in how IDPDs compare with  factorization properties studied
by Anderson, Anderson, and Zafrullah~\cite{AAZ}. 
Here is a summary diagram of our findings in this direction; see Section~\ref{sec170725b} for definitions.
(The un-labeled implications/non-implications in this diagram are well-known.)
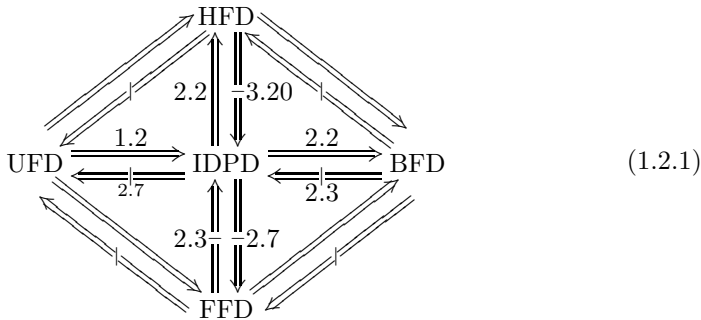
\begin{equation}
\label{diag170909a}
\begin{split}\xymatrix@=15mm{
&\text{HFD}
\ar@{=>}[ld]|{|}
\ar@<2ex>@{=>}[rd]
\ar@<1ex>@{=>}[d]|{-}^{\ref{thm170908a}}
\\
\text{UFD}
\ar@<1ex>@{=>}[r]^{\ref{ex170725a}}
\ar@<2ex>@{=>}[ru]
\ar@{=>}[rd]
&\text{IDPD}
\ar@<1ex>@{=>}[l]|{|}^{\text{\ref{ex170725b}}}
\ar@<1ex>@{=>}[u]^{\ref{thm170725a}}
\ar@<1ex>@{=>}[d]|{-}^{\ref{ex170725b}}
\ar@<1ex>@{=>}[r]^{\ref{thm170725a}}
&\text{BFD}
\ar@{=>}[lu]|{|}
\ar@<2ex>@{=>}[ld]|{|}
\ar@<1ex>@{=>}[l]|{|}^{\ref{ex170725d}}
\\
&\text{FFD}
\ar@<2ex>@{=>}[lu]|{|}
\ar@<1ex>@{=>}[u]|{-}^{\ref{ex170725d}}
\ar@{=>}[ru]
}
\end{split}
\end{equation}
In the process of completing this diagram, we show that many familiar examples of HFDs are also IDPDs.
(This actually forms the bulk of the paper.)
For instance, we prove the following.
\begin{enumerate}[(a)]
\item\label{itemthm170725b}
Let $K$ be a field, and let $F\subseteq K$ be a subfield.
Then the ring $D=F+XK[X]$ is an IDPD. 
(Theorem~\ref{thm170725b})
\item\label{itemthm170730b}
The ring $\bbz[\sqrt{-3}]$ is an IDPD. 
Moreover, it is the unique, non-integrally closed imaginary quadratic IDPD.
(Theorem~\ref{thm170730b})
\item\label{itemthm170903b}
Let $p$ be a prime number, and
let $D$ be a Krull domain with divisor class group $\Cl(D)\cong\bbz_{p^\infty}$
or $\Cl(D)\cong\bbz_{p^k}$ for some $k\in\bbn$.
If $D$ is an HFD, then $D$ is an IDPD.
(Theorem~\ref{thm170903b})
\end{enumerate}

As to the organization of the paper, Section~\ref{sec170725b} looks at IDPDs in general.
In particular, this section fills in most of the above summary diagram and discusses the localization behavior of IDPDs.

In Section~\ref{sec170906a}, we focus on Krull domains (in particular, Dedekind domains).
Here we use the divisor class group to get at the relation between IDPDs and 
HFDs (half factorial domains)
For instance, we show that several classes of HFDs
from the seminal paper of Zaks~\cite{zaks} are in fact IDPDs, 
and it is here in Theorem~\ref{thm170906a} that we exhibit a Dedekind domain that is an HFD but not an IDPD.

Lastly, Appendix~\ref{sec170903a} contains three technical lemmas for use in the proofs of Section~\ref{sec170906a}.
We relegate them to the appendix since they deal only with properties of natural numbers.

\section{General IDPDs} \label{sec170725b}

In this section, we investigate IDPDs that are not necessarily Krull domains.
We begin with the following prime-free characterization of IDPDs.

\begin{prop}\label{prop190826a}
An atomic  domain
$D$ is an IDPD
if and only if the following condition holds:
\begin{enumerate}[$(\dagger)$]
\item for every non-zero non-unit $z\in D$ that is a product of non-prime atoms and for every pair $p,q\in D$ of non-associate 
non-prime irreducible divisors
of $z$, there exist  (non-prime) atoms $p',q'\in D$ such that $pp'\sim qq'\mid z$.
\end{enumerate}
\end{prop}

\begin{proof}
The forward implication is trivial. For the converse, assume that $D$ is an atomic  domain satisfying condition ($\dagger$),
and let a
non-zero non-unit $z\in D$ be given with  $p,q\in D$  non-associate irreducible divisors
of $z$.
Then there are elements $a,b\in D$ such that $pa=z=qb$.
If $p$ or $q$ is prime, then we are done as in Example~\ref{ex170725a}.
So, assume that $p$ and $q$ are non-prime. 

Let $z=p_1\cdots p_n$ be an irreducible factorization of $z$
and reorder the factors if necessary to assume that $p_i$ is prime if and only if $i\leq N$ for some integer $N\geq 0$. 
We induct on $N$.
The base case $N=0$ is covered by the assumption~($\dagger$).
For the inductive step, assume that $N\geq 1$, so in particular $p_1$ is prime.
We have 
\begin{equation}\label{eq190829a}
p_1\cdots p_n=z=pa=qb
\end{equation}
so the fact that $p_1$ is prime and $p,q$ are not implies that $p_1\mid a$ and $p_1\mid b$,
say $a=p_1a_1$ and $b=p_1b_1$.
Substitute these in~\eqref{eq190829a} and cancel $p_1$ to find

Note that the logic of Example~\ref{ex170725a} shows that 
$z_1:=p_2\cdots p_n=pa_1=qb_1$.
Now apply the inductive hypothesis to conclude that there exist atoms $p',q'\in D$ such that $pp'\sim qq'\mid z_1\mid z$, as desired.

Lastly, let $z\in D$ be a 
non-zero non-unit  that is a product of non-prime atoms, and let $p,q\in D$ be non-associate 
non-prime irreducible divisors
of $z$, and let $p',q'\in D$ be atoms such that $pp'\sim qq'\mid z$.
We need to show that $p',q'$ are non-prime. By way of contradiction, if $p'$ were prime, then it would be a prime factor of $z$, 
and the logic of Example~\ref{ex170725a} shows that every irreducible factorization of $z$ has a prime element in the list of factors, contradicting our
assumptions for $z$. So, $p'$ is not prime, and neither is $q'$ by symmetry.
\end{proof}

We continue with the following
part of diagram~\eqref{diag170909a} that also motivates much of Section~\ref{sec170906a}.
For this result, recall that an atomic domain $D$ is a \emph{half-factorial domain (HFD)} 
if it satisfies the following half of the defining property for a UFD:
for all atoms $p_1,\ldots,p_m,q_1,\ldots,q_n\in D$ if $p_1\cdots p_m=q_1\cdots q_n$, then $m=n$.
Also, recall that the integral domain $D$ is a \emph{bounded factorization domain (BFD)} if it is atomic
and for each non-zero non-unit $z\in D$ there is a bound on the lengths of the irreducible factorizations of $z$ in $D$.

\begin{thm}\label{thm170725a}
If $D$ is an IDPD, then it is an HFD, in particular, it is a BFD.
\end{thm}

\begin{proof}
Assume that $D$ is an IDPD. 
By definition, this implies that $D$ is atomic.
To show that $D$ is an HFD, let $z\in D$ be a non-zero non-unit with  irreducible factorizations
\begin{equation}
\label{eq170725a}
p_1\cdots p_m=z=q_1\cdots q_n.
\end{equation}
We need to show that $m=n$. 
Assume without loss of generality that $m\geq n$.

We argue by induction on $n$.
The base case $n=1$ is straightforward since the $p_i$ and $q_i$ are atoms.

For the induction step,  assume that $m\geq n\geq 2$ and that given irreducible factorizations 
$\pi_1\cdots\pi_\ell=\tau_1\cdots\tau_{k}$ with $k<n$ or $\ell<n$, then one has $\ell=k$.
If there are integers $i,j$ such that $p_i\sim q_j$, then re-order the factors in~\eqref{eq170725a} to assume that $p_m\sim q_n$;
in this case, there is a unit $u\in U(D)$ such that $p_1\cdots p_{m-1}=(uq_1)q_2\cdots q_{n-1}$,
so the induction hypothesis implies that $n-1=m-1$, so $n=m$ as desired.

Assume for the rest of the proof that $p_i\not\sim q_j$ for all $i,j$. 
In particular, $p_1\not\sim q_1$.
The fact that $D$ is an IDPD then implies that there are atoms $p',q'$ such that
$p_1p'\sim q_1q'\mid z$. 
Thus, we have $p_1p'=vq_1q'$ for some $v\in U(D)$. We  replace $q'$ with the atom $vq'$ to assume that
$p_1p'= q_1q'\mid z$. 

Now, write $p_1p'z'=z$ for some $z'\in D\ssm\{0\}$.
It follows that we have
$p_1p'z'=p_1\cdots p_m$, so 
\begin{equation}
\label{eq170725b}
p'z'=p_2\cdots p_m.
\end{equation}
Similarly, we have 
\begin{equation}
\label{eq170725c}
q'z'=q_2\cdots q_n.
\end{equation}
Since $D$ is atomic, there are atoms $\xi_1,\ldots,\xi_\ell\in D$ and a unit $w\in U(D)$ such that $\ell\geq 0$ and 
$z'=w\xi_1\cdots\xi_\ell$. (The unit is included so that we don't have to treat the case $\ell=0$ separately.)
Equations~\eqref{eq170725b} and~\eqref{eq170725c} then imply that
\begin{align}
p_2\cdots p_m&=wp'\xi_1\cdots\xi_\ell
\label{eq170725d}
\\
q_2\cdots q_n&=wq'\xi_1\cdots\xi_\ell
\label{eq170725e}
\end{align}
The  equality~\eqref{eq170725e} here, with our induction hypothesis, implies that $\ell+1=n-1$.
In particular, we have $\ell+1<n$, so equation~\eqref{eq170725d} implies that $m-1=\ell+1=n-1$, so $m=n$ as desired.
\end{proof}

For the next example, 
which is part of diagram~\eqref{diag170909a},
recall that the integral domain $D$ is a \emph{finite factorization domain (FFD)} if it is atomic
and for each non-zero non-unit $z\in D$ there are only  finitely many
distinct irreducible factorizations (up to associates) of $z$ in $D$.

\begin{ex}\label{ex170725d}
Let $k$ be a field.
Then the subring $k[X^2,X^3]\subseteq k[X]$ is an FFD (hence a BFD) that is not an HFD;
see, e.g., \cite[Sections~3 and~5]{AAZ}.
Essentially, this ring is not an HFD because $(X^2)^3=(X^3)^2$.
In particular, Theorem~\ref{thm170725a} implies that this ring is not an IDPD.
\end{ex}

Most of the rest of this section deals with classes of HFDs that are also IDPDs. 
In part, we do this to show how close HFDs are to IDPDs, and to show that many of the standard examples of HFDs are IDPDs,
that you have to work a bit to find an example of an HFD that is not an IDPD.
Also, we do this to show that our example of an HFD that is not an IDPD is minimal in some sense;
see Remark~\ref{disc170908a}.

This is~\eqref{itemthm170725b} from the introduction.

\begin{thm}\label{thm170725b}
Let $K$ be a field, and let $F\subseteq K$ be a subfield.
Then the subring $D=F+XK[X]\subseteq K[X]$ is an IDPD. 
\end{thm}

\begin{proof}
The ring $D$ is atomic by~\cite[Theorem~2.9]{AAZ2}.
Furthermore, this result shows that the atoms of $D$ are of the following form:
\begin{enumerate}[(1)]
\item $aX$ where $a\in K^\times$, and
\item\label{thm170725b2} $a(1+Xf)$ where $a\in F^\times$ and $f\in K[X]$ and $1+Xf$ is an atom in $K[X]$.
\end{enumerate}
In addition, $D$ is an HFD by~\cite[Theorem~5.3]{AAZ2}, and every atom from~\eqref{thm170725b2} is prime
by~\cite[Theorem~2.9]{AAZ2}.

To show that $D$ is an IDPD, we  show that condition~($\dagger$) from Proposition~\ref{prop190826a}
is satisfied vacuously. 
To this end, let $z\in D$ be a non-zero non-unit that is a product of non-prime atoms, and let $p,q\in D$ be non-associate 
non-prime irreducible divisors
of $z$. The preceding paragraph shows that $p$ and $q$ are of the form $aX$ and $bX$, respectively, where $a,b\in K^\times$.
It follows that $p$ and $q$ are associates in $D$, contradicting our assumptions.
\end{proof}

\begin{disc}\label{disc170802a}
The proof of the HFD version of Theorem~\ref{thm170725b} is quite different from our proof here. 
The point for HFDs seems to be the following straightforward fact:
If $D$ is an atomic subring of a domain $R$ such that $R$ is an HFD such that
every atom of $D$ is also an atom in $R$, then $D$ is also an HFD. 
One then shows that $D=F+XK[X]\subseteq K[X]$ is an HFD
by considering it as a subring of the UFD (hence, HFD) $K[X]$, using the characterization of atoms of $D$
from~\cite[Theorem~2.9]{AAZ2}.
On the other hand, it is not clear to us at this moment that one can detect the IDPD property similarly. 
Thus, we pose the following.
\end{disc}

\begin{question} \label{q170730d}
If $D$ is an atomic subring of a domain $R$ such that $R$ is an IDPD and 
every atom of $D$ is an atom in $R$, must $D$  also be an IDPD?
\end{question}

Here is another part of diagram~\eqref{diag170909a}.

\begin{ex}\label{ex170725b}
The rings $\bbq+X\bbr[X]$ and $\bbr+X\bbc[X]$ are not UFDs essentially because 
$(2X)(\frac 12X)=X^2$ and
$(iX)(-iX)=X^2$, respectively.
Also, these rings are not FFDs because $X^2=(\eta X)(\frac 1\eta X)$ has infinitely many non-associate divisors.
But these rings are IDPDs by Theorem~\ref{thm170725b}.

Note that this example also shows that IDPDs need not be integrally closed in their fields of fractions, since
neither $\bbq+X\bbr[X]$ nor $\bbr+X\bbc[X]$ is integrally closed.
Indeed, the quotient field of $\bbq+X\bbr[X]$ is $\bbr(X)$ since $r=(rX)/X$, but $\sqrt 2\in \bbr(X)$ is integral
over $\bbq+X\bbr[X]$ and not in $\bbq+X\bbr[X]$.
A similar argument works for $\bbr+X\bbc[X]$ using $i\in\bbc(X)\ssm\bbr+X\bbc[X]$.
Moreover, these rings are not Krull domains; see, e.g., \cite[p.~4]{AAZ}.

Moreover, this example shows that the IDPD property need not ascend along polynomial extensions. 
Indeed, If $D[T]$ is an IDPD, then it is an HFD, so we know from~\cite[Theorem~5.4]{coykendall:ehfd} that $D$ is integrally closed.
In particular, this shows that the examples from the previous paragraph are IDPDs whose polynomial extensions $D[T]$ are not IDPDs. 
\end{ex}

\begin{question} \label{q170730b}
Is there a version of Theorem~\ref{thm170725b} that fully characterizes the IDPD property for rings of the form $A+XB[X]$
or more generally $D+M$? 
What about for integer-valued polynomials or for pullbacks as in our work with 
Boynton~\cite{boynton:rpb,boynton:ccfrp}?
\end{question}

\begin{disc}\label{disc170727a}
If $D$ is an IDPD, it is natural to ask whether 
for every non-zero non-unit $z\in D$ and for every pair $p,q\in D$ of non-associate irreducible divisors
of $z$, must we have $pq\mid z$? The answer is
definitely not.
Indeed, if it did, then~\cite[Theorem~5.1]{coykendall:idg} would prove that $D$ is a UFD, which is not true in general (see,
e.g., Example~\ref{ex170725b}).
\end{disc}

The next result of this section builds on~\cite[Theorem~2.3]{coykendall:hfdqf}.
It is item~\eqref{itemthm170730b} from the introduction.

\begin{thm}\label{thm170730b}
The ring $\bbz[\sqrt{-3}]$ is an IDPD. 
Moreover, it is the unique, non-integrally closed imaginary quadratic IDPD.
\end{thm}

\begin{proof}
The uniqueness comes from~\cite[Theorem~2.3]{coykendall:hfdqf}: if $D$ is a 
non-integrally closed imaginary quadratic IDPD, then it is a
non-integrally closed imaginary quadratic HFD, so $D \cong\bbz[\sqrt{-3}]$ by \emph{loc.\ cit.}
Thus, it remains to show that $\bbz[\sqrt{-3}]$ is an IDPD.
Recall that $\bbz[\sqrt{-3}]$ is an HFD by~\cite[Section 3, Example]{zaks}. 

A standard fact from number theory states that, up to associates, the non-prime atoms of $D$ are 2, $1+\sqrt{-3}$, and $1-\sqrt{-3}$.
Moreover, they satisfy the  equations
\begin{gather*}
2^2=(1+\sqrt{-3})(1-\sqrt{-3})\\
(1+\sqrt{-3})^2=2(1-\sqrt{-3})\qquad
(1-\sqrt{-3})^2=-2(1+\sqrt{-3}).
\end{gather*}
From this, and by symmetry, Proposition~\ref{prop190826a} shows that we need only consider $p=2$ and $q=1+\sqrt{-3}$, along with  three
cases for $z$ (all other cases are multiples of these): 

Case 1: $z=2^2=4$. In this case, use $p'=2$ and $q'=1-\sqrt{-3}$ to verify the defining property for IDPDs.

Case 2: $z=(1+\sqrt{-3})^2=2(1-\sqrt{-3})$. Use $p'=1-\sqrt{-3}$ and $q'=1+\sqrt{-3}$.

Case 3: $z=(1-\sqrt{-3})^2=-2(1+\sqrt{-3})$. Use $p'=1+\sqrt{-3}$ and $q'=2$.
\end{proof}

We end this section with some localization results for IDPDs, based on~\cite{MR1190405}.

\begin{defn}\label{defn190901a}
A saturated multiplicatively closed subset $U\subseteq D$ is \emph{splitting} if for each $x\in D$ there are elements
$a\in D$ and $s\in U$ such that $x=as$ and $aD\cap tD=atD$ for all $t\in U$. We call such a factorization an \emph{AAZ-factorization} of $a$,
after the authors of~\cite{MR1190405}.
\end{defn}

\begin{ex}\label{fact190901a}
If $D$ is atomic and $U$ is a saturated multiplicatively closed subset generated by prime elements of $D$, then $U$ is splitting by~\cite[Corollary~1.7]{MR1190405}.
Moreover, the proof of \emph{loc.\ cit.} shows that any $x\in D$ has a factorization of the form $x=as$ where $s\in U$ and $a$ is not divisible by any prime in $U$, and moreover any such factorization is an AAZ-factorization.
\end{ex}

Our first localization result is a version of~\cite[Theorem~2.1]{MR1190405} for IDPDs.

\begin{thm}\label{thm190901a}
Let $D$ be an IDPD and $U$ a splitting multiplicatively closed subset of $D$. Then $U^{-1}D$ is an IDPD.
\end{thm}

\begin{proof}
By~\cite[Theorem~2.1]{MR1190405}, the ring $U^{-1}D$ is atomic.
As in Definition~\ref{defn170725a},
let $x$ be a non-zero non-unit of $U^{-1}D$,
and let $p$ and $q$ be 
non-associate irreducible factors of $x$ in $U^{-1}D$. 
Clear denominators (so multiply by units) if necessary to assume that $x,p,q\in D$.

Fix AAZ-factorizations 
$p=as$ and $q=bt$.
By~\cite[Corollary~1.4(c)]{MR1190405}, the elements $a,b\in D$ are 
atoms of $D$. 
Moreover, if $a\sim b$ in $D$, then $p\sim a\sim b\sim q$ in $U^{-1}D$, contradicting our assumptions; thus $a$ and $b$ are not associates in $D$. 

Now, the elements $p$ and $a$ are associates in $U^{-1}D$,
so the condition $p\mid x$ in $U^{-1}D$ implies that $a\mid x$ in $U^{-1}D$.
It follows that there is an element $u\in U$ such that $a\mid ux$.
Similarly, there is an element $v\in U$ such that $b\mid vx$, so we have
$a,b\mid uvx$. Since $D$ is an IDPD, there are atoms $a',b'\in D$ such that $aa'\sim bb'\mid uvx$ in $D$,
hence $aa'\sim bb'\mid uvx\sim x$ in $U^{-1}D$.
Since $pa'\sim aa'\sim bb'\sim qb'$, it remains to show that $a'$ and $b'$ are atoms in $U^{-1}D$.

By~\cite[Lemma~1.1 and Proposition~1.5]{MR1190405} it suffices to show that $a'$ and $b'$ are not units in $U^{-1}D$.
Suppose by way of contradiction (and by symmetry) that $a'$ is a unit in $U^{-1}D$.
It follows that $a\sim aa'\sim bb'$ in $U^{-1}D$. Since $a$ and $b$ are atoms in $U^{-1}D$, it follows that $b'$ is also a unit in $U^{-1}D$.
Thus, we have $a\sim aa'\sim bb'\sim b$ in $U^{-1}D$, 
contradicting our assumptions.
\end{proof}

Our final localization result is an IDPD analog of~\cite[Theorem~3.1]{MR1190405}.

\begin{thm}\label{thm190901b}
Let  $U$ be a saturated multiplicatively closed subset generated by prime elements of $D$ such that $U^{-1}D$ is an IDPD.
Then $D$ is an IDPD.
\end{thm}

\begin{proof}
By~\cite[Theorem~3.1]{MR1190405}, the domain $D$ is atomic.
As in Proposition~\ref{prop190826a}, let $x$ be a non-zero non-unit of $D$ that is a product of non-prime atoms,
and let $p,q$ be non-associate, non-prime irreducible divisors of $x$ in $D$. In particular, since $U$ is generated by prime elements of $D$,
we have $p,q\notin U$, so $p,q$ are non-units in $U^{-1}D$. 
From~\cite[Lemma~1.1 and Proposition~1.9]{MR1190405}, the elements $p,q$ are irreducible in $U^{-1}D$.
Also, Example~\ref{fact190901a} shows that $p=p\cdot 1$ is an AAZ-factorization of $p$, and similarly for $q$,
so $p,q$ are not prime in $U^{-1}D$ by~\cite[Corollary~1.4(c)]{MR1190405}.

We claim that $p\not\sim q$ in $U^{-1}D$.
By way of contradiction, suppose that $p\sim q$ in $U^{-1}D$.
It follows that there are elements $u,v\in U$ such that $vp=uq$ in $D$.
Since $U$ is generated by primes and $p,q$ are not prime, it is straightforward to show that $v\sim u$ and $p\sim q$ in $D$, contradicting our assumptions.
This establishes the claim.

Now, apply the IDPD property for $U^{-1}D$ to find atoms $p',q'\in U^{-1}D$ such that $pp'\sim qq'\mid x$ in $U^{-1}D$.
Clear denominators and use AAZ-factorizations as in the proof of Theorem~\ref{thm190901a} to assume that
$p',q'$ are atoms in $D$.

We claim that $p'$ and $q'$ are not prime in $D$.
Suppose by way of contradiction (and by symmetry) that $p'$ is prime in $D$. The condition
$p'\mid x$ in $U^{-1}D$ implies that there is an element $w\in U$ such that $p'\mid wx$ in $D$.
If $p'\mid w$, then the fact that $U$ is saturated implies that $p'\in U$, so $p'$ is a unit in $U^{-1}D$, a contradiction.
Thus, the fact that $p'$ is prime implies that $p'\mid x$. However, our assumptions on $x$ imply that $x$ has no prime factors.
This final contradiction establishes our second claim.

Next, we show that $pp'\sim qq'$ in $D$. The condition $pp'\sim qq'$ in $U^{-1}D$ implies that there are 
elements $u,v\in U$ such that $upp'=vqq'$ in $D$. Since $p,p',q,q'$ are non-prime atoms in $D$ and $U$ is generated by primes,
it follows readily that in $D$ we have $u\mid v$ and $v\mid u$ so $u\sim v$. From this, we obtain $vpp'\sim upp'=vqq'$ 
and hence $pp'\sim qq'$ in $D$.

To complete the proof, we show that $qq'\mid x$ in $D$. Since $qq'\mid x$ in $U^{-1}D$, there are elements
$y\in D$ and $w\in U$ such that $qq'y=wx$. Since $q,q'$ are non-prime atoms and $w$ is a product of primes, it follows that $w\mid y$,
say $y=wz$ and so $wx=qq'y=qq'wz$. Thus, we have $x=qq'z$, i.e., $qq'\mid x$ in $D$, as desired.
\end{proof}

\section{Krull Domains, Dedekind Domains, HFDs, and IDPDs}\label{sec170906a}

In this section, we further investigate the relation between the HFD and IDPD property,
focusing on
Krull domains (e.g., of Dedekind domains)
where we can use the divisor class group to understand factorizations.
Accordingly, we begin here by summarizing some facts about  divisor class groups.
See the text of Fossum~\cite{fossum:dcgkd} for general properties
and the paper of Zaks~\cite{zaks} for relevant connections with HFDs.

\begin{disc}\label{disc170803a}
Let $D$ be a Krull domain, and let $\Cl(D)$ denote the divisor class group of $D$.
For divisorial ideals $\fa$ and $\fb$ of $D$, we write $\fa\equiv\fb$ if $\fa$ and $\fb$ represent the same class in $\Cl(D)$.
Note that Secton~1, line~9 of~\cite{AAZ}  says that $D$ is atomic.

For each non-zero non-unit $x\in D$, the ideal $xD$ can be decomposed as a finite intersection
of symbolic powers of height-1 primes of $D$, say $xD=\bigcap_{\ell=1}^L\q_\ell^{(e_\ell)}$ with each $e_\ell\geq 1$.
Zaks~\cite[Lemma~5.2]{zaks} points out that if $xD=\q^{(e)}$ and $e$ is the order of the class of $\q$ in $\Cl(D)$,
then $x$ is an atom. If, for instance, $\Cl(D)$ is an elementary 2-group, then we have $e\leq 2$ here.
Furthermore, from~\cite[Proposition~6.2]{zaks}, we know 
that $x$ is an atom if and only if there are no proper sub-intersections of the decomposition
$\bigcap_{\ell=1}^L\q_\ell^{(e_\ell)}$ that are principal. 

If $y$ is another non-zero non-unit of $D$ and $x$ is not necessarily an atom,
then we can combine decompositions of $xD$ and $yD$ to find a list of distinct height-1 primes $\p_1,\ldots,\p_T$
and integers $a_t,b_t\geq 0$ such that $xD=\bigcap_{t=1}^T\p_t^{(a_t)}$ and $yD=\bigcap_{t=1}^T\p_t^{(b_t)}$.
In this case, we have $xyD=\bigcap_{t=1}^T\p_t^{(a_t+b_t)}$ by~\cite[Proposition~7.2]{zaks}.
Similar expressions hold for finite products $x_1\cdots x_n$.
Furthermore, a straightforward localization argument shows that
$\bigcap_{t=1}^T\p_t^{(a_t)}=\bigcap_{t=1}^T\p_t^{(b_t)}$ if and only if $a_t=b_t$ for all $t$; 
moreover, we have
$\bigcap_{t=1}^T\p_t^{(a_t)}\subseteq\bigcap_{t=1}^T\p_t^{(b_t)}$ if and only if $a_t\geq b_t$ for all $t$.
\end{disc}

Some readers may wish to skip the proof of the next result on the first reading.

\begin{lem}\label{lem170903a}
Let $D$ be a Krull domain  such that $\Cl(D)$ is a (possibly infinite) direct sum of cyclic groups
$\phi\colon\Cl(D)\xra\cong\bbz^{(I)}\oplus\left(\bigoplus_{j\in J}\bbz/n_j\bbz\right)$ with each $n_j\geq 2$.
Let $\mathcal I$ denote the standard basis for $\bbz^{(I)}$, and let $\mathcal J$ denote the standard generating sequence for
$\bigoplus_{j\in J}\bbz/n_j\bbz$, both considered as subsets of $\bbz^{(I)}\oplus\left(\bigoplus_{j\in J}\bbz/n_j\bbz\right)$.
Set $-\mathcal I=\{-e\in\bbz^{(I)}\mid e\in\mathcal I\}$, and let $X$ denote the set of non-principal height-1 prime ideals of $D$.
Assume that the function $\psi\colon X\to\bbz^{(I)}\oplus\left(\bigoplus_{j\in J}\bbz/n_j\bbz\right)$ induced by the isomorphism $\phi$
has $\im(\psi)=\mathcal I\cup-\mathcal I\cup\mathcal J$.
If $D$ is an HFD, then $D$ is an IDPD.
\end{lem}

\begin{proof}
Assume that $D$ is an HFD.
Set $G=\bbz^{(I)}\oplus\left(\bigoplus_{j\in J}\bbz/n_j\bbz\right)$.
We use the following notation for the height-1 primes of $D$.
The primes in $\psi^{-1}(\mathcal I)$ will be denoted using the symbol $\p$, possibly with indices, e.g., as $\p_i$ or $\p_{i,j}$.
The primes in $\psi^{-1}(-\mathcal I)$ will be denoted similarly using the symbol $\q$,
the primes in $\psi^{-1}(\mathcal J)$ will be denoted using the symbol $\fr$,
and the principal primes (if there are any) will be denoted using the symbol $\fs$.
For convenience, let $\mathcal K$ be an index set for  the (possibly empty) set of principal prime ideals of $D$.
Using this notation, the elements of $\Cl(D)$ are all represented as finite intersections of the form
\begin{equation}\label{eq170903a}
\fa=
\left(\bigcap_{i\in\mathcal I}\bigcap_{p_i}\p_{i,p_i}^{(a_{i,p_i})}\right)
\cap
\left(\bigcap_{i\in\mathcal I}\bigcap_{q_i}\q_{i,q_i}^{(b_{i,q_i})}\right)
\cap
\left(\bigcap_{j\in\mathcal J}\bigcap_{r_j}\fr_{j,r_j}^{(c_{j,r_j})}\right)
\cap
\left(\bigcap_{k\in\mathcal K}\fs_{k}^{(d_{k})}\right).
\end{equation}
Because of the arithmetic of $G$, the ideal $\fa$ is principal if and only if
for each $i\in\mathcal I$ and each $j\in\mathcal J$ we have
$\sum_{p_i}a_{i,p_i}=\sum_{q_i}b_{i,q_i}$ and $n_j\mid\sum_{r_j}c_{j,r_j}$.
Using Zaks' sub-intersection criterion from Remark~\ref{disc170803a}, we see that the principal ideals generated by 
atoms of $D$ are precisely those of the following form:
\begin{enumerate}[(1)]
\item \label{item170903a}
$\p_i\cap\q_i$ for some $i\in\mathcal I$,
\item \label{item170903b}
$\bigcap_{r_j}\fr_{j,r_j}^{(c_{j,r_j})}$ for some $j\in\mathcal J$ such that $\sum_{r_j}c_{j,r_j}=n_j$, or
\item \label{item170903c}
$\fs_k$.
\end{enumerate}
In particular, given a non-zero non-unit $z\in D$, write $\fa=zD$ as in~\eqref{eq170903a} and recall that $D$ is atomic;
the above description of the atoms of $D$ shows that the number of irreducible factors in an irreducible factorization of $z$ is exactly
$\sum_{i\in\mathcal I}\sum_{p_i}a_{i,p_i}+\sum_{j\in\mathcal J}\frac{1}{n_j}\sum_{r_j}c_{j,r_j}+\sum_{k\in\mathcal K}d_k$.
Thus, $D$ is an HFD. 

Let $\ell\colon D\ssm\{0\}\to\bbn_0=\{n\in\bbz\mid n\geq 0\}$ be given by
$\ell(a)=n$ where $a=up_1\cdots p_n$ in $D$ with $u$ being a unit and each $p_i$ being an atom.
This is well-defined since $D$ is an HFD.
For instance, given $a\in D\ssm\{0\}$ we have $\ell(a)=0$ if and only if $a$ is a unit, and $\ell(a)=1$ if and only if $a$ is an atom. 
Furthermore, we have $\ell(ab)=\ell(a)+\ell(b)$ for all $a,b\in D\ssm\{0\}$.

By Proposition~\ref{prop190826a}, let $z\in D$ be a non-zero non-unit which is a product of non-prime atoms, 
and let $p_1,q_1$ be non-prime non-associate irreducible factors of $z$ in $D$. 
To show that these elements satisfy the defining property in~\ref{defn170725a},
we argue by induction on $n=\ell(z)$.
In the base case $\ell(z)=1$, the element $z$ is an atom with irreducible factors $p_1$ and $q_1$.
It follows that $p_1\sim z\sim q_1$, contradicting our non-associate assumption on $p_1$ and $q_1$. 
For good measure, we therefore consider a second base case $\ell(z)=2$. In this case,
since $p_1$ and $q_1$ are irreducible factors of $z$, there are atoms $p_2,q_2\in D$ such that
$p_1p_2=z=q_1q_2$. Thus, the defining property in~\ref{defn170725a} is satisfied in this case.

For the induction step, assume that $n=\ell(z)>2$ and that the result holds for non-zero non-units $z'\in D$ with $\ell(z')=n-1$.
Write $p_1\cdots p_n=z=q_1\cdots q_n$ for some atoms $p_i$ and $q_i$. Since $z$ has a factorization as a product of non-prime atoms,
the ideas of Example~\ref{ex170725a} show that 
no $p_i$ is prime, nor is any $q_i$ prime. 
Thus, when we write $zD$ in the form~\eqref{eq170903a}
we have
$$
zD=
\left(\bigcap_{i\in\mathcal I}\bigcap_{p_i}\p_{i,p_i}^{(a_{i,p_i})}\right)
\cap
\left(\bigcap_{i\in\mathcal I}\bigcap_{q_i}\q_{i,q_i}^{(b_{i,q_i})}\right)
\cap
\left(\bigcap_{j\in\mathcal J}\bigcap_{r_j}\fr_{j,r_j}^{(c_{j,r_j})}\right).
$$
Furthermore, the specific form of the atoms given in the first paragraph of this proof show that
$p_1D$ is either of the form 
$\p_{i_1,p_{i_1}}\cap\q_{i_1,q_{i_1}}$ for some $i_1\in\mathcal I$, or
$\bigcap_{r_{j_1}}\fr_{j_1,r_{j_1}}^{(d_{j_1,r_{j_1}})}$ for some $j_1\in\mathcal J$ such that $\sum_{r_{j_1}}d_{j_1,r_{j_1}}=n_{j_1}$.
Similarly, 
$q_1D$ is either of the form 
$\p_{s_1,p_{s_1}}\cap\q_{s_1,q_{s_1}}$ for some $s_1\in\mathcal I$, or
$\bigcap_{r_{t_1}}\fr_{t_1,r_{t_1}}^{(e_{t_1,r_{t_1}})}$ for some $t_1\in\mathcal J$ such that $\sum_{r_{t_1}}e_{t_1,r_{t_1}}=n_{t_1}$.

Case 1: 
$p_1D=\p_{i_1,p_{i_1}}\cap\q_{i_1,q_{i_1}}$ for some $i_1\in\mathcal I$
and $q_1D=\bigcap_{r_{t_1}}\fr_{t_1,r_{t_1}}^{(e_{t_1,r_{t_1}})}$ for some $t_1\in\mathcal J$ such that 
$\sum_{r_{t_1}}e_{t_1,r_{t_1}}=n_{t_1}$.
The condition $p_1\mid z$ implies that 
$$\left(\bigcap_{i\in\mathcal I}\bigcap_{p_i}\p_{i,p_i}^{(a_{i,p_i})}\right)
\cap
\left(\bigcap_{i\in\mathcal I}\bigcap_{q_i}\q_{i,q_i}^{(b_{i,q_i})}\right)
\cap
\left(\bigcap_{j\in\mathcal J}\bigcap_{r_j}\fr_{j,r_j}^{(c_{j,r_j})}\right)
=zD\subseteq p_1D=\p_{i_1,p_{i_1}}\cap\q_{i_1,q_{i_1}}$$
and the condition $q_1\mid z$ implies that 
$$\left(\bigcap_{i\in\mathcal I}\bigcap_{p_i}\p_{i,p_i}^{(a_{i,p_i})}\right)
\cap
\left(\bigcap_{i\in\mathcal I}\bigcap_{q_i}\q_{i,q_i}^{(b_{i,q_i})}\right)
\cap
\left(\bigcap_{j\in\mathcal J}\bigcap_{r_j}\fr_{j,r_j}^{(c_{j,r_j})}\right)
=zD\subseteq q_1D=\bigcap_{r_{t_1}}\fr_{t_1,r_{t_1}}^{(e_{t_1,r_{t_1}})}.$$
Since the primes $\p_{i_1,p_{i_1}}$ and $\q_{i_1,q_{i_1}}$ do not occur in the list of $\fr_{t_1,r_{t_1}}$'s,
it follows that 
\begin{align*}
zD
&=
\textstyle
\left(\bigcap_{i\in\mathcal I}\bigcap_{p_i}\p_{i,p_i}^{(a_{i,p_i})}\right)
\cap
\left(\bigcap_{i\in\mathcal I}\bigcap_{q_i}\q_{i,q_i}^{(b_{i,q_i})}\right)
\cap
\left(\bigcap_{j\in\mathcal J}\bigcap_{r_j}\fr_{j,r_j}^{(c_{j,r_j})}\right)
\\
&\subseteq
\textstyle
\left(\p_{i_1,p_{i_1}}\cap\q_{i_1,q_{i_1}}\right)\cap\left(\bigcap_{r_{t_1}}\fr_{t_1,r_{t_1}}^{(e_{t_1,r_{t_1}})}\right)
\\
&=p_1q_1D.
\end{align*}
We conclude that $p_1q_1\mid z$,
so the defining property in~\ref{defn170725a} is satisfied in this case.

Case 2: 
$p_1D=\bigcap_{r_{t_1}}\fr_{t_1,r_{t_1}}^{(d_{t_1,r_{t_1}})}$ for some $t_1\in\mathcal J$ such that 
$\sum_{r_{t_1}}d_{t_1,r_{t_1}}=n_{t_1}$
and $q_1D=\p_{i_1,p_{i_1}}\cap\q_{i_1,q_{i_1}}$ for some $i_1\in\mathcal I$.
This case is handled as in Case~1?.

Case 3: 
$p_1D=\p_{i_1,p_{i_1}}\cap\q_{i_1,q_{i_1}}$ for some $i_1\in\mathcal I$,
and $q_1D=\p_{s_1,u_{s_1}}\cap\q_{s_1,v_{s_1}}$ for some $i_1\in\mathcal I$.
If $i_1\neq s_1$, then the list of primes $\p_{i_1,p_{i_1}},\q_{i_1,q_{i_1}},\p_{s_1,u_{s_1}},\q_{s_1,v_{s_1}}$ has no repetitions;
in this event, it follows as in Case~1 that $p_1q_1\mid z$ and we are done.
Thus, we assume without loss of generality that
$p_1D=\p_{i_1,p_{i_1}}\cap\q_{i_1,q_{i_1}}$ 
and $q_1D=\p_{i_1,u_{i_1}}\cap\q_{i_1,v_{i_1}}$ for some $i_1\in\mathcal I$.
If there are no repetitions in the 
list of primes $\p_{i_1,p_{i_1}},\q_{i_1,q_{i_1}},\p_{i_1,u_{i_1}},\q_{i_1,v_{i_1}}$
then we are done again as in Case~1, so assume that there is a repetition in the list.
That is, assume that $\p_{i_1,p_{i_1}}=\p_{i_1,u_{i_1}}$ or $\q_{i_1,q_{i_1}}=\q_{i_1,v_{i_1}}$.
We have assumed that $p_1$ and $q_1$ are not associate, i.e., $p_1D\neq q_1D$, so 
we cannot have both
$\p_{i_1,p_{i_1}}=\p_{i_1,u_{i_1}}$ and $\q_{i_1,q_{i_1}}=\q_{i_1,v_{i_1}}$.

By symmetry, we assume without loss of generality that $\p_{i_1,p_{i_1}}=\p_{i_1,u_{i_1}}$ and $\q_{i_1,q_{i_1}}\neq\q_{i_1,v_{i_1}}$.
Therefore, we have $a_{i_1,p_{i_1}},b_{i_1,q_{i_1}},b_{i_1,v_{i_1}}\geq 1$
so $\sum_{t_{i_1}}a_{i_1,t_{i_1}}=\sum_{u_{i_1}}b_{i_1,u_{i_1}}\geq 2$.
It follows that either $a_{i_1,p_{i_1}}\geq 2$ or there is a $w_{i_1}\neq p_{i_1}$ such that $a_{i_1,w_{i_1}}\geq 1$.
If $a_{i_1,p_{i_1}}\geq 2$, then as in Case~1 we have
\begin{align*}
zD
&=
\textstyle
\left(\bigcap_{i\in\mathcal I}\bigcap_{p_i}\p_{i,p_i}^{(a_{i,p_i})}\right)
\cap
\left(\bigcap_{i\in\mathcal I}\bigcap_{q_i}\q_{i,q_i}^{(b_{i,q_i})}\right)
\cap
\left(\bigcap_{j\in\mathcal J}\bigcap_{r_j}\fr_{j,r_j}^{(c_{j,r_j})}\right)
\\
&\subseteq
\textstyle
\p_{i_1,p_{i_1}}^{(2)}\cap\q_{i_1,q_{i_1}}\cap\q_{i_1,v_{i_1}}
\\
&=p_1q_1D
\end{align*}
and we are done.
On the other hand, if there is a $w_{i_1}\neq p_{i_1}$ such that $a_{i_1,w_{i_1}}\geq 1$, then we have
\begin{align*}
zD
&=
\textstyle
\left(\bigcap_{i\in\mathcal I}\bigcap_{p_i}\p_{i,p_i}^{(a_{i,p_i})}\right)
\cap
\left(\bigcap_{i\in\mathcal I}\bigcap_{q_i}\q_{i,q_i}^{(b_{i,q_i})}\right)
\cap
\left(\bigcap_{j\in\mathcal J}\bigcap_{r_j}\fr_{j,r_j}^{(c_{j,r_j})}\right)
\\
&\subseteq
\textstyle
\p_{i_1,p_{i_1}}\cap\p_{i_1,w_{i_1}}\cap\q_{i_1,q_{i_1}}\cap\q_{i_1,v_{i_1}}.
\end{align*}
Note that the ideal
$\p_{i_1,w_{i_1}}\cap\q_{i_1,v_{i_1}}$
is principal and generated by an atom, say 
$\p_{i_1,w_{i_1}}\cap\q_{i_1,v_{i_1}}=p'D$.
Continuing the previous display, we have
\begin{align*}
zD
&\subseteq
\textstyle
(\p_{i_1,p_{i_1}}\cap\q_{i_1,q_{i_1}})\cap(\p_{i_1,w_{i_1}}\cap\q_{i_1,v_{i_1}})
=(p_1D)(p'D)
=p_1p'D
\end{align*}
so the atom $p'$ satisfies $p_1p'\mid z$.
Similarly, there is an atom $q'\in D$ such that
$\p_{i_1,w_{i_1}}\cap\q_{i_1,q_{i_1}}=q'D$
and as above we have
\begin{align*}
zD
&\subseteq
\textstyle
(\p_{i_1,p_{i_1}}\cap\q_{i_1,v_{i_1}})\cap(\p_{i_1,w_{i_1}}\cap\q_{i_1,q_{i_1}})
=(q_1D)(q'D)
=q_1q'D
\end{align*}
so the atom $q'$ satisfies $q_1q'\mid z$
and furthermore
$$q_1q'D=\p_{i_1,p_{i_1}}\cap\p_{i_1,w_{i_1}}\cap\q_{i_1,q_{i_1}}\cap\q_{i_1,v_{i_1}}=p_1p'D$$
so $q_1q'\sim p_1p'$.
This completes Case~3.

Case~4:
$p_1D=\bigcap_{r_{j_1}}\fr_{j_1,r_{j_1}}^{(d_{j_1,r_{j_1}})}$ for some $j_1\in\mathcal J$ such that $\sum_{r_{j_1}}d_{j_1,r_{j_1}}=n_{j_1}$,
and
$q_1D=\bigcap_{r_{t_1}}\fr_{t_1,r_{t_1}}^{(e_{t_1,r_{t_1}})}$ for some $t_1\in\mathcal J$ such that $\sum_{r_{t_1}}e_{t_1,r_{t_1}}=n_{t_1}$.
As in Case~3, we assume without loss of generality that $t_i=j_i$,
so 
$p_1D=\bigcap_{r_{j_1}}\fr_{j_1,r_{j_1}}^{(d_{j_1,r_{j_1}})}$
and
$q_1D=\bigcap_{r_{j_1}}\fr_{j_1,r_{j_1}}^{(e_{j_1,r_{j_1}})}$
with $\sum_{r_{j_1}}d_{j_1,r_{j_1}}=n_{j_1}=\sum_{r_{j_1}}e_{j_1,r_{j_1}}$.

Claim: there is an index $r'_{j_1}$ such that $d_{j_1,r'_{j_1}}<e_{j_1,r'_{j_1}}$.
Suppose that no such index exists. 
Then $d_{j_1,r_{j_1}}\geq e_{j_1,r_{j_1}}$ for all $r_{j_1}$, so we have
$$p_1D=\bigcap_{r_{j_1}}\fr_{j_1,r_{j_1}}^{(d_{j_1,r_{j_1}})}
\subseteq\bigcap_{r_{j_1}}\fr_{j_1,r_{j_1}}^{(e_{j_1,r_{j_1}})}=q_1D.$$
since $p_1$ and $q_1$ are both atoms, it follows that $p_1D=q_1D$, contradicting the assumption $p_1\not\sim q_1$.

For each index $r_{j_1}$, set $M_{j_1,r_{j_1}}=\max(d_{j_1,r_{j_1}},e_{j_1,r_{j_1}})\leq c_{j_1,r_{j_1}}$.
(We are going to use a version of the standard LCM construction here.)
The above claim implies that there is an index $r'_{j_1}$ such that $d_{j_1,r'_{j_1}}<M_{j_1,r'_{j_1}}$.
Recall that $\sum_{r_{j_1}}c_{j_1,r_{j_1}}$ is divisible by $n_{j_1}$, say $\sum_{r_{j_1}}c_{j_1,r_{j_1}}=xn_{j_1}$.
It follows that
$$n_{j_1}
=\sum_{r_{j_1}}d_{j_1,r_{j_1}}
<\sum_{r_{j_1}}M_{j_1,r_{j_1}}
\leq\sum_{r_{j_1}}c_{j_1,r_{j_1}}
=xn_{j_1}
$$
so we have $x\geq 2$.
Furthermore, we have 
$M_{j_1,r_{j_1}}\leq d_{j_1,r_{j_1}}+e_{j_1,r_{j_1}}$,
so 
$$
\sum_{r_{j_1}}M_{j_1,r_{j_1}}
\leq\sum_{r_{j_1}}d_{j_1,r_{j_1}}+\sum_{r_{j_1}}e_{j_1,r_{j_1}}
=2n_{j_1}
\leq xn_{j_1}
=\sum_{r_{j_1}}c_{j_1,r_{j_1}}.
$$
Since only finitely many of the $c_{j_1,r_{j_1}}$ are non-zero,
Lemma~\ref{lem170903b}
provides integers $y_{j_1,r_{j_1}}\in\bbn_0$ such that
$y_{j_1,r_{j_1}}\leq c_{j_1,r_{j_1}}-M_{j_1,r_{j_1}}$ for all $r_{j_1}$ and such that 
$\sum_{r_{j_1}}M_{j_1,r_{j_1}}+\sum_{r_{j_1}}y_{j_1,r_{j_1}}=2n_{j_1}$.

For each  $r_{j_1}$, set $\delta_{r_{j_1}}=(M_{r_{j_1}}-d_{r_{j_1}})+y_{r_{j_1}}\geq 0$
and $\epsilon_{r_{j_1}}=(M_{r_{j_1}}-e_{r_{j_1}})+y_{r_{j_1}}\geq 0$.
Note that 
$$\sum_{r_{j_1}}\delta_{r_{j_1}}=\sum_{r_{j_1}}(M_{r_{j_1}}+y_{r_{j_1}})-\sum_{r_{j_1}}d_{r_{j_1}}=2n_{j_1}-n_{j_1}=n_{j_1}.$$
Thus, using the arithmetic in $G$, we see that the ideal
$\bigcap_{r_{j_1}}\fr_{j_1,r_{j_1}}^{(\delta_{j_1,r_{j_1}})}$ 
is principal generated by an atom $p'\in D$.
Similarly, the ideal
$\bigcap_{r_{j_1}}\fr_{j_1,r_{j_1}}^{(\epsilon_{j_1,r_{j_1}})}$ 
is principal generated by an atom $q'\in D$.
We will be done with this case and the proof of the lemma once we show that $p_1p'\sim q_1q'\mid z$.

From the arithmetic in $\Cl(D)$, we have
\begin{align*}
p_1p'D
&=\bigcap_{r_{j_1}}\fr_{j_1,r_{j_1}}^{(d_{j_1,r_{j_1}}+\delta_{j_1,r_{j_1}})}
=\bigcap_{r_{j_1}}\fr_{j_1,r_{j_1}}^{(M_{j_1,r_{j_1}}+y_{j_1,r_{j_1}})}
=\bigcap_{r_{j_1}}\fr_{j_1,r_{j_1}}^{(e_{j_1,r_{j_1}}+\epsilon_{j_1,r_{j_1}})}
=q_1q'D
\end{align*}
hence $p_1p'\sim q_1q'$.
Furthermore, the condition $y_{j_1,r_{j_1}}\leq c_{j_1,r_{j_1}}-M_{j_1,r_{j_1}}$
implies that $M_{j_1,r_{j_1}}+y_{j_1,r_{j_1}}\leq c_{j_1,r_{j_1}}$,
so
\begin{align*}
zD
&=
\textstyle
\left(\bigcap_{i\in\mathcal I}\bigcap_{p_i}\p_{i,p_i}^{(a_{i,p_i})}\right)
\cap
\left(\bigcap_{i\in\mathcal I}\bigcap_{q_i}\q_{i,q_i}^{(b_{i,q_i})}\right)
\cap
\left(\bigcap_{j\in\mathcal J}\bigcap_{r_j}\fr_{j,r_j}^{(c_{j,r_j})}\right)
\\
&\subseteq
\textstyle
\bigcap_{r_{j_1}}\fr_{j_1,r_{j_1}}^{(M_{j_1,r_{j_1}}+y_{j_1,r_{j_1}})}\\
&=q_1q'D
\end{align*}
and thus 
$p_1p'\sim q_1q'\mid z$, as desired.
\end{proof}

Next, we document some consequences of Lemma~\ref{lem170903a}.

\begin{prop}\label{prop170903a}
The Dedekind domain HFDs constructed in the proof of~\cite[Theorem~3.4]{chapman:hfds} are IDPDs.
\end{prop}

\begin{proof}
By construction, these rings satisfy the hypotheses of Lemma~\ref{lem170903a}.
\end{proof}

Here is a version of~\cite[Theorem~3.4]{zaks} for our setting.
Note that the hypotheses of \emph{loc.\ cit.} are slightly incorrect. 
Indeed, Zaks assumes that for every height-1 prime ideal $\p$ of $D$, there is a height-1 prime ideal $\q$ of $D$
such that $\p\cap\q\equiv D$.
What he means to assume (as one sees in his proof and in his applications) is that
for every height-1 prime ideal $\p$ of $D$, either $\p^{(2)}\equiv D$ or there is a height-1 prime ideal $\q$ of $D$
such that $\p\cap\q\equiv D$.

\begin{thm}\label{cor170903a}
Let $D$ be a Krull domain such that for every height-1 prime ideal $\p$ of $D$, either $\p^{(2)}\equiv D$ or 
there is a height-1 prime ideal $\q$ of $D$ such that $\p\cap\q\equiv D$.
If $D$ is an HFD, then $D$ is an IDPD.
\end{thm}

\begin{proof}
Assume that $D$ is an HFD.
The proof of~\cite[Theorem~3.4]{zaks} shows that our assumptions imply that 
$\Cl(D)$ is the direct sum of an elementary 2-group $E$ and a free abelian group.
In particular, $E$ is a direct sum of copies of $\bbz/2\bbz$. 
Furthermore, the proof of  \emph{loc.\ cit.} shows that the hypotheses of our Lemma~\ref{lem170903a} are satisfied,
so we conclude from our lemma that $D$ is an IDPD. 
\end{proof}

\begin{cor}\label{thm170802a}
Let $D$ be a Krull domain  such that $\Cl(D)$ is an elementary 2-group.
If $D$ is an HFD, then $D$ is an IDPD.
\end{cor}

\begin{proof}
Immediate from Theorem~\ref{cor170903a}.
\end{proof}

\begin{cor}\label{cor170802a}
Let $D$ be a Krull domain.
If $|\Cl(D)|\leq 2$, then
$D$ is an IDPD.
\end{cor}

\begin{proof}
Zaks~\cite[Theorem~1.4]{zaks} tells us that $D$ is an HFD, 
so it is an IDPD by Corollary~\ref{thm170802a}.
\end{proof}

\begin{cor}\label{cor170803a}
Let $D$ be a Krull domain.
Then the following  are equivalent:
\begin{enumerate}[\rm(i)]
\item\label{cor170803a1}
$D[x]$ is an IDPD, 
\item\label{cor170803a2}
$D[x]$ is an HFD, and
\item\label{cor170803a3}
$|\Cl(D)|\leq 2$.
\end{enumerate}
\end{cor}

\begin{proof}
The implication~\eqref{cor170803a1}$\implies$\eqref{cor170803a2} is from Theorem~\ref{thm170725a}.
Also, Zaks~\cite[Theorem~2.4]{zaks} 
gives the implication~\eqref{cor170803a2}$\implies$\eqref{cor170803a3}.

\eqref{cor170803a3}$\implies$\eqref{cor170803a1}
Assume that $|\Cl(D)|\leq 2$. Zaks~\cite[Theorem~2.4]{zaks}
implies that $D[x]$ is an HFD. 
Recall that $D[x]$ is also a Krull domain such that $\Cl(D[x])\cong\Cl(D)$;
see, e.g., \cite[Proposition~1.6 and Theorem~8.1]{fossum:dcgkd}.
Thus, $D[x]$ is an IDPD by Corollary~\ref{thm170802a}.
\end{proof}

Next, we treat some Dedekind domains;
see the subsequent corollary for some rings of integers.
It is worth noting that one can prove this result directly, without resorting to
Corollary~\ref{thm170802a}, using ideas from~\cite[Carlitz's Theorem]{chapman:hfds} and~\cite[Theorem~2.4]{coykendall:ehfd}.

\begin{cor}\label{thm170730a}
Let $D$ be a Dedekind domain.
If $|\Cl(D)|\leq 2$, then
$D$ is an IDPD; the converse holds if $\Cl(D)$ is torsion and $D$ has the property that there is a prime ideal in each class.
\end{cor}

\begin{proof}
One implication is from Corollary~\ref{cor170802a} since Dedekind domains are noetherian and integrally closed, hence Krull domains.
For the converse statement, assume that
$D$ is an IDPD such that $\Cl(D)$ is torsion and $D$ has the property that there is a prime ideal in each class.
Then we have $|\Cl(D)|\leq 2$ by~\cite[Theorem~2.4]{coykendall:ehfd}.
\end{proof}

\begin{cor}\label{thm170725c}
Let $K$ be a number field, and let $D$ be the corresponding ring of integers, i.e., $K$ is a finite extension
of $\bbq$ and $D$ is the integral closure of
$\bbz$ in $K$.
Then 
$D$ is an IDPD if and only if $D$ is an HFD, that is, if and only if $K$ has class number 1 or 2.
\end{cor}

\begin{proof}
A result of Carlitz~\cite{carlitz} says that
$D$ is an HFD if and only if $K$ has class number 1 or 2, that is, if and only if $|\Cl(D)|\leq 2$;
see also~\cite[Carlitz's Theorem]{chapman:hfds}. 
If $D$ is an IDPD, then Theorem~\ref{thm170725a} implies that $D$ is an HFD.
For the converse, if $D$ is an HFD, then $D$ is a Dedekind domain with $|\Cl(D)|\leq 2$,
so $D$ is an IDPD by Corollary~\ref{thm170730a}.
\end{proof}

\begin{ex}\label{ex170725c}
The ring $\bbz[\sqrt{-5}]$ is not a UFD, essentially because $(2)(3)=6=(1+\sqrt{-5})(1-\sqrt{-5})$.
But this ring is an IDPD by Corollary~\ref{thm170725c}.
\end{ex}

Our next results give more Krull domains where HFD implies IDPD. 
First, some notation from~\cite{zaks}.

\begin{notn}\label{notn170805a}
Let $D$ be a Krull domain, and assume that $\Cl(D)$ is torsion.
Each $\fa\in\Cl(D)$ has the form $\fa=\bigcap_{i=1}^e\p_i^{(c_i)}$ where each $\p_i$ is a height-1 prime and $c_i\in\bbn_0$.
Write $L(\fa)=\sum_{i=1}^ec_i/o(\p_i)$ where $o(\p_i)$ is the order of $\p_i$ in $\Cl(D)$.
(This is where we use the torsion assumption on $\Cl(D)$.)
For each non-zero non-unit $x\in D$, set $\ell(x)=L(xD)$.
\end{notn}

\begin{disc}\label{disc170805a}
Let $D$ be a Krull domain, and assume that $\Cl(D)$ is torsion.
Zaks~\cite{zaks} notes that $\ell$ is a function from the set of non-zero non-units of $D$ to the positive rational numbers. 
In~\cite[Theorem~3.3]{zaks}, he shows that if 
for all non-zero non-units $x\in D$ we have $\ell(x)=1$ if and only if $x$ is an atom in $D$,
then $D$ is an HFD. (It is worth noting that, in particular, the condition $\ell(x)=1$ if and only if $x$ is an atom in $D$ implies that
$\im(\ell)\subseteq\bbz$. Indeed, Remark~\ref{disc170803a} implies that $D$ is atomic, so every non-zero non-unit $z\in D$,
we have an irreducible factorization $z=p_1\cdots p_n$, so $\ell(z)=\ell(p_1)+\cdots+\ell(p_n)=n\in\bbz$.
Example~\ref{ex170809a}  shows, however, that without the assumption on atoms
one need not have $\im(\ell)\subseteq\bbz$.)
Conversely, he shows that if $D$ is an HFD, then 
for all non-zero non-units $x\in D$ we have $\ell(x)=1$ if and only if $x$ is an atom in $D$.
In particular, when $D$ is an HFD, if $\fa$ is a principal ideal of $D$, then $L(\fa)\in\bbn$; Example~\ref{ex170805a} below
shows that the converse of this
statement fails in general.
This is good to keep in mind in the proofs of the subsequent results.
\end{disc}

\begin{ex}\label{ex170809a}
A result of Gilmer, Heinzer, and Smith~\cite[Theorem~5]{MR1434383} implies that there is a 
Dedekind domain $D$ with $\Cl(D)\cong\bbz_{3}^2$
such that the non-principal prime ideals of $D$ correspond 
precisely to the elements $(1,0),(0,1),(1,2)\in\bbz_{3}^2$.
In other words,  each non-principal prime ideal of $D$ corresponds to either $(1,0)$ or 
$(0,1)$ or $(1,2)$ in a fixed isomorphism  $\phi\colon\Cl(D)\xra\cong\bbz_{3}^2$,
and furthermore, there exist non-principal prime ideals $\p$, $\q$, and $\fr$ in $D$ correspond to 
$(1,0)$, $(0,1)$, and $(1,2)$, respectively. 
Computing orders in $\Cl(D)$, we have 
$o(\p)=o(\q)=o(\fr)=3$.
In particular, we have $L(\p^{(a)}\cap\q^{(b)}\cap\fr^{(c)})=\frac 13(a+ b+ c)$ for all $a,b,c\in\bbn_0$.

Now, the ideal $\fa=\p^{(2)}\cap\q\cap\fr$ is principal since in $\bbz_{3}^2$ it corresponds to $2(1,0)+(0,1)+(1,2)=(0,0)$.
(Moreover, using the same type of computation, it is straightforward to show that no proper sub-intersection of
$\p^{(2)}\cap\q\cap\fr$ is principal, so Remark~\ref{disc170803a} tells us that  $\fa=\pi D$ for some atom $\pi\in D$.)
On the other hand, using Zak's notation, we have $l(\pi)=L(\fa)=\frac 13(2+1+1)=4/3\notin\bbz$.
\end{ex}

\begin{ex}\label{ex170805a}
Let $D$ be a Krull domain such that $\Cl(D)$ is a non-cyclic elementary 2-group, and assume that $D$ is an HFD;
see, e.g., \cite[Example~9]{zaks} for examples.
Note that $D$ has two non-principal height-1 primes $\p$ and $\q$
such that $\p\not\equiv\q$ in $\Cl(D)$. 
Indeed, if this were not true, the fact that $\Cl(D)$ is generated by the height-1 primes would imply that $\Cl(D)$ is cyclic,
contradicting our assumption.

Now, the fact that $\p$ and $\q$ are not principle conspires with the fact that $\Cl(D)$ is an elementary 2-group
to show that $o(\p)=2=o(\q)$. Thus, $\p$ is its own inverse in $\Cl(D)$,
so the condition $\p\not\equiv\q$ implies that $\p\cap\q$ is not principle. On the other hand, we have
$L(\p\cap\q)=(1/2)+(1/2)=1$, as desired.
\end{ex}

In our next result we set $\bbz_n=\bbz/n\bbz$, and 
let 
$\bbz_{p^\infty}$ denote the direct limit of the directed system
$\bbz_p\into\bbz_{p^2}\into\bbz_{p^3}\into\cdots$.

This is item~\eqref{itemthm170903b} from the introduction.

\begin{thm}\label{thm170903b}
Let $p$ be a prime number, and
let $D$ be a Krull domain with $\Cl(D)$ isomorphic to $\bbz_{p^\infty}$
or $\bbz_{p^k}$ for some $k\in\bbn$.
If $D$ is an HFD, then $D$ is an IDPD.
\end{thm}

\begin{proof}
Assume that $D$ is an HFD. 
Zaks~\cite[Theorem~8.4]{zaks} tells us that for all height-1 prime ideals $\p$ and $\q$ of $D$, there is an integer $m\in\bbn$ such that
either $\p=\q^{(p^m)}$ or $\q=\p^{(p^m)}$.
From this, we have the following: for every finite set $X$ of height-1 prime ideals, 
there is an integer $I\in\bbn_0$ such that
one can partition $X$ into a union of pairwise disjoint subsets $X_0,\ldots,X_I$ such that
(a) for all $\p\in X_i$ and all $\q\in X_j$ with $j\leq i$ we have $\q\equiv\p^{(p^{i-j})}$, and 
(b) for all $\p\in X_i$ the order of $\p$ in $\Cl(D)$ is $o(\p)=p^i$.

Let $\fa\subsetneq R$ be a divisorial ideal of $D$.
Then $\fa$ has a decomposition as a finite intersection of symbolic powers $\p^{(a(\p))}$ of height-1 prime ideals with exponents $a(\p)\geq 1$. 
Let $X$ be the set of prime ideals that occur in this decomposition. 
Partition the set $X$ as in the preceding paragraph, then we have
\begin{equation}\label{eq170903b}
\fa=\bigcap_{i=0}^I\bigcap_{\p\in X_i}\p^{(a(\p))}.
\end{equation}
Assume without loss of generality that $X_I\neq\emptyset$, and let $\q\in X_J$ for some $J\geq I$.
Condition~(a) from the preceding paragraph implies that
$$
\fa
=\bigcap_{i=0}^I\bigcap_{\p\in X_i}\p^{(a(\p))}
=\q^{(\sum_{i=0}^Ip^{J-i}\sum_{\p\in X_i}a(\p))}.
$$
In particular, since $o(\q)=p^J$ in $\Cl(D)$, the ideal $\fa$ is principal if and only if
$\sum_{i=0}^Ip^{J-i}\sum_{\p\in X_i}a(\p)\in p^J\bbz$, that is, if and only if $L(\fa)\in\bbz$,
where $L$ is as in Notation~\ref{notn170805a}. (One 
of these implications is from Zaks~\cite[Theorem~3.3]{zaks} and, as we have noted in Remark~\ref{disc170805a} it holds for any HFD with torsion
class group. 
The other implication, though, is special to this situation, as we saw in Example~\ref{ex170805a}.)
Furthermore, if $\fa$ is principal, then it is generated by an atom of $D$ if and only if $L(\fa)=1$,
i.e., if and only if $\sum_{i=0}^Ip^{J-i}\sum_{\p\in X_i}a(\p)= p^J$.

Let  $\ell$ be as in Notation~\ref{notn170805a}.
Let $z\in D$ be a non-zero non-unit, and let $\pi,\tau$ be non-associate irreducible factors of $z$ in $D$. 
To show that $\pi$ and $\tau$ satisfy the defining property in~\ref{defn170725a},
set $r=\ell(z)$. 
As in the proof of Lemma~\ref{lem170903a}, note that~$r\geq 2$.

Decompose the ideal $\fa=zD$ as in~\eqref{eq170903b}
$$zD=\bigcap_{i=0}^I\bigcap_{\p\in X_i}\p^{(a(\p))}
$$
for some $a(\p)\in\bbn$.
Assume without loss of generality that $X_I\neq\emptyset$, and let $\q\in X_I$.
Our above analysis shows that 
$\sum_{i=0}^Ip^{I-i}\sum_{\p\in X_i}a(\p)= rp^I$ where $r=\ell(z)\in\bbn$.

Our assumptions on $z$, $\pi$, and $\tau$ imply that $zD\subseteq\pi D\cap\tau D$.
Using the decomposition of $zD$ above, it follows that 
for all $\p\in X$ there are $b(\p),c(\p)\in\bbn_0$ with $b(\p),c(\p)\leq a(\p)$ and such that
$\pi D=\bigcap_{i=0}^I\bigcap_{\p\in X_i}\p^{(b(\p))}$
and
$\tau D=\bigcap_{i=0}^I\bigcap_{\p\in X_i}\p^{(c(\p))}$.
Since $\pi$ and $\tau$ are atoms, our above analysis shows that 
$$\sum_{i=0}^Ip^{I-i}\sum_{\p\in X_i}b(\p)= p^I=\sum_{i=0}^Ip^{I-i}\sum_{\p\in X_i}c(\p).$$

For each $\p\in X$, let $M(\p)=\max(b(\p),c(\p))\leq a(\p)$.
As in the proof of Lemma~\ref{lem170903a} Case~4, 
we have
\begin{align*}
\sum_{i=0}^Ip^{I-i}\sum_{\p\in X_i}M(\p)
&\leq \sum_{i=0}^Ip^{I-i}\sum_{\p\in X_i}b(\p)+\sum_{i=0}^Ip^{I-i}\sum_{\p\in X_i}c(\p)
\\
&=2p^I
\\
&\leq rp^I
\\
&=\sum_{i=0}^Ip^{I-i}\sum_{\p\in X_i}a(\p).
\end{align*}
Lemma~\ref{lem170903c}
implies that there is a function $f\colon X\to\bbn_0$ such that $M(\p)+f(\p)\leq a(\p)$ for all $\p\in X$
and such that $\sum_{i=0}^Ip^{I-i}\sum_{\p\in X_i}(M(\p)+f(\p))= 2p^I$.

Define $\beta,\gamma\colon X\to\bbn_0$ as
$\beta=(M-b)+f$
and
$\gamma=(M-c)+f$.
Again, as in the proof of Lemma~\ref{lem170903a} Case~4, 
we have
\begin{align*}
\sum_{i=0}^Ip^{I-i}\sum_{\p\in X_i}\beta(\p)
&=\sum_{i=0}^Ip^{I-i}\sum_{\p\in X_i}(M(\p)+f(\p))
-\sum_{i=0}^Ip^{I-i}\sum_{\p\in X_i}b(\p)
\\
&=2p^I-p^I
\\
&=p^I
\intertext{and similarly}
\sum_{i=0}^Ip^{I-i}\sum_{\p\in X_i}\gamma(\p)
&=p^I.
\end{align*}
Our analysis at the beginning of this proof shows that the ideals
$\bigcap_{i=0}^I\bigcap_{\p\in X_i}\p^{(\beta(\p))}$
and
$\bigcap_{i=0}^I\bigcap_{\p\in X_i}\p^{(\gamma(\p))}$
are principal and generated by atoms:
$\bigcap_{i=0}^I\bigcap_{\p\in X_i}\p^{(\beta(\p))}=\pi'D$
and
$\bigcap_{i=0}^I\bigcap_{\p\in X_i}\p^{(\gamma(\p))}=\tau'D$.
Furthermore, we have
\begin{align*}
\pi\pi'D
&=\bigcap_{i=0}^I\bigcap_{\p\in X_i}\p^{(b(\p)+\beta(\p))}\\
&=\bigcap_{i=0}^I\bigcap_{\p\in X_i}\p^{(M(\p)+f(\p))}\\
&=\bigcap_{i=0}^I\bigcap_{\p\in X_i}\p^{(c(\p)+\gamma(\p))}\\
&=\tau\tau'D
\end{align*}
so $\pi\pi'\sim\tau\tau'$.
Also, the condition $M(\p)+f(\p)\leq a(\p)$ for all $\p\in X$ implies that
\begin{align*}
zD
&=\bigcap_{i=0}^I\bigcap_{\p\in X_i}\p^{(a(\p))}
\subseteq
\bigcap_{i=0}^I\bigcap_{\p\in X_i}\p^{(c(\p)+\gamma(\p))}
=\tau\tau'D
\end{align*}
so $\pi\pi'\sim\tau\tau'\mid z$.
Thus, the defining property in~\ref{defn170725a} is satisfied, so $D$ is an IDPD, as desired.
\end{proof}

The proof of the next result is  similar to the previous one.
However, we require a bit of notation.

\begin{notn}\label{notn170906a}
Let  $D$ be a Krull domain such that $\Cl(D)$ is cyclic  of order $n\in\bbn$
and there is a height-1 prime ideal $\p$ whose class in $\Cl(D)$ generates $\Cl(D)$.
By assumption, there is an isomorphism $\phi\colon\bbz_n\xra\cong\Cl(D)$ such that $\phi([1])$ is the class of $\p$ in $\Cl(R)$. 
For any $s\in\bbz$ and any divisorial ideal $\fa$, we write $s\approx \fa$ to mean that $\phi([s])$ is the class of $\fa$ in $\Cl(D)$.
Set
$$S(D)=\{s\in\bbz\mid\text{$1\leq s\leq n$ and $s\approx\q$ for some height-1 prime ideal $\q$}\}.$$
We often set $S=S(D)$ when it does not result in any ambiguity.
\end{notn}

The set $S(D)$ in Notation~\ref{notn170906a} appears to depend on the choice of isomorphism $\phi$. 
However, one consequence of the next result is that, when $D$ is an HFD, it does not depend on the choice of $\phi$;
indeed, the result shows that there is at most one height-1 prime ideal representing an element of order $n$ in
$\Cl(D)\cong\bbz_n$.

\begin{lem}\label{prop170909a}
Let  $D$ be a Krull domain such that $\Cl(D)$ is cyclic of order $n$. 
Assume that $D$ is an HFD. 
Then for each divisor $r\mid n$ there is at most one element of $\Cl(D)$ of order $r$
represented by a height-1 prime ideal.
\end{lem}

\begin{proof}
Suppose that $\p$ and $\q$ are 
height-1 prime ideals  of $D$ such that $\p\not\equiv\q$ in $\Cl(D)$ and
$o(\p)=r=o(\q)$.
Since $\Cl(D)$ is cyclic of order $n$, the fact that $\p$ and $\q$ have the same order
implies that they generate the same subgroup of $\Cl(D)$, necessarily with order $r$. 
In particular, there is an integer $x\in\bbn$ such that $\q\equiv\p^{(x)}$ in $\Cl(D)$ and such that
$1<x<r$.

Consider the ideal $\fa=\p^{(n-x)}\cap\q$. Using the arithmetic in $\Cl(D)$, we have
$\fa=\p^{(n-x)}\cap\q\equiv\p^{(n-x+x)}\equiv D$, so $\fa$ is principal. 
We work with Notation~\ref{notn170805a}:
$$L(\fa)=\frac{n-x}{r}+\frac{1}{r}=\frac{n+1-x}{r}.
$$
Since $D$ is an HFD and $\fa$ is principal, Remark~\ref{disc170805a} says that $L(\fa)\in\bbz$.
Thus, the divisibility condition $r\mid |\Cl(D)|=n$ implies that
$r\mid(x-1)$.
Because we have $x>1$, hence $x-1>0$, it follows that $x-1\geq r$, so $x>r$.
But we have assumed that $x<r$, which provides the desired contradiction.
\end{proof}

The following lemma is motivated by~\cite[Lemma~4.7]{MR1858160}.

\begin{lem}\label{lem170906c}
Let  $D$ be a Krull domain such that $\Cl(D)$ is cyclic and there is a height-1 prime ideal $\p$ whose class in $\Cl(D)$ generates $\Cl(D)$.
If $S=S(D)$ is as in Notation~\ref{notn170906a},
then every $s\in S$ satisfies $s\mid n$.
\end{lem}

\begin{proof}
We work in the setting of Notation~\ref{notn170906a}.
Let $s\in S=S(D)$, and let $\q$ be a height-1 prime ideal of $D$ such that $s\approx\q$.
Set $\delta=n/o(\q)$ where $o(\q)$ is the order of the class of $\q$ in $\Cl(D)$.
Consider the ideal $\fa=\p^{(n-s)}\cap\q$.
By assumption, we have $\q\equiv\p^{(s)}$ in $\Cl(D)$, so $\fa\equiv\p^{(n)}\equiv R$.
That is, $\fa$ is principal.
Since $\fa$ is an HFD and a Krull domain with torsion divisor class group, 
Remark~\ref{disc170805a} implies that $L(\fa)\in\bbz$.
However, we have
$$L(\fa)=\frac{n-s}{n}+\frac{1}{o(\q)}
=1-\frac sn+\frac\delta n
=1+\frac{\delta-s}{n}$$
so we must have $n\mid(\delta-s)$.
Since $\delta$ and $s$ lie between 1 and $n$ (inclusive), uniqueness of remainders modulo $n$ implies that $s=\delta\mid n$, as desired.
\end{proof}

The next result recovers the case $\Cl(D)\cong\bbz_{p^n}$ of Theorem~\ref{thm170903b} because of~\cite[Theorem~8.4]{zaks}.

\begin{thm}\label{thm170906a}
Let  $D$ be a Krull domain such that $\Cl(D)$ is cyclic and there is a height-1 prime ideal $\p$ whose class in $\Cl(D)$ generates $\Cl(D)$.
Assume that the set $S$ from Notation~\ref{notn170906a} is totally ordered by divisibility.
If $D$ is an HFD, then $D$ is an IDPD.
\end{thm}

\begin{proof}
Assume that $D$ is an HFD.
Again, we work under the assumptions and notation from~\ref{notn170906a}.
From our assumptions on $D$, we have the following: for every finite set $Y$ of height-1 prime ideals, 
one can partition $Y$ into a union $\bigcup_{s\in S}Y_s$ of pairwise disjoint subsets $Y_s$ such that
(a) for all $\fs\in Y_s$ and all $\ft\in Y_t$ with $t\mid s$ (i.e., $t\leq s$) we have $\fs\equiv\ft^{(s/t)}$, and 
(b) for all $\fs\in Y_s$ the order of $\fs$ in $\Cl(D)$ is $o(\fs)=n/s$.
In particular, item~(a) here implies that $\fs=\p^{(s)}$ for all $\fs\in Y_s$.

Let $\fa\subsetneq R$ be a divisorial ideal of $D$.
Then $\fa$ has a decomposition as a finite intersection of symbolic powers $\fs^{(a(\fs))}$ of height-1 prime ideals with exponents $a(\fs)\geq 1$. 
Let $Y$ be the set of primes that occur in this decomposition. 
Partition the set $T$ as in the preceding paragraph, then we have
\begin{equation}\label{eq170906a}
\fa=\bigcap_{s\in S}\bigcap_{\fs\in Y_s}\fs^{(a(\fs))}.
\end{equation}
Condition~(a) from the preceding paragraph implies that
$$
\fa
=\bigcap_{s\in S}\bigcap_{\fs\in Y_s}\fs^{(a(\fs))}
=\p^{(\sum_{s\in S}s\sum_{\fs\in Y_s}a(\fs))}.
$$
In particular, since $o(\p)=n$ in $\Cl(D)$, the ideal $\fa$ is principal if and only if
$\sum_{s\in S}s\sum_{\fs\in Y_s}a(\fs)\in n\bbz$, that is, if and only if $L(\fa)\in\bbz$,
where $L$ is as in Notation~\ref{notn170805a}.
Furthermore, if $\fa$ is principal, then it is generated by an atom of $D$ if and only if $L(\fa)=1$,
that is, if and only if $\sum_{s\in S}s\sum_{\fs\in Y_s}a(\fs)= n$.

Let $\ell$ be as in Notation~\ref{notn170805a}.
Let $z\in D$ be a non-zero non-unit, and let $\pi,\tau$ be non-associate irreducible factors of $z$ in $D$. 
To show that $\pi$ and $\tau$ satisfy the defining property in~\ref{defn170725a},
set $r=\ell(z)$. 
As in the proof of Lemma~\ref{lem170903a}, note that~$r\geq 2$.

Decompose the ideal $\fa=zD$ as in~\eqref{eq170906a}
$$zD=\bigcap_{s\in S}\bigcap_{\fs\in Y_s}\fs^{(a(\fs))}
$$
for some $a(\fs)\in\bbn$.
Our above analysis shows that 
$\sum_{s\in S}s\sum_{\fs\in Y_s}a(\fs)= rn$ where $r=\ell(z)\in\bbn$.
The argument concludes as in the proof of Theorem~\ref{thm170903b},
with an application of Lemma~\ref{lem170906a} in place of Lemma~\ref{lem170903c}.
\end{proof}

One common feature of some of the preceding proofs is  the implication 
$L(\fa)\in\bbz\implies\fa\equiv D$.
One may be tempted to think that, given  a Krull domain $D$ with torsion divisor class group 
such that $L(\fa)\in\bbz\implies\fa\equiv D$, if $D$ is an HFD then it is an IDPD.
The proof of our next result shows this to be false in general.
Furthermore, this result provides the final piece of diagram~\eqref{diag170909a}.

\begin{thm}\label{thm170908a}
There exists a Dedekind domain $D$ such that $\Cl(D)\cong\bbz_6$ and such that $D$ is an HFD but not an IDPD.
\end{thm}

\begin{proof}
The result~\cite[Theorem~5]{MR1434383} provides a
Dedekind domain $D$ such that $\Cl(D)\cong\bbz_6$ and such that 
$S=S(D)=\{1,2,3\}$ in the notation of~\ref{notn170906a}.
Moreover, from~\cite[Theorem~8]{MR1434383}, we may assume further that there are infinitely many
height-1  prime (i.e., maximal) ideals 
in each of the classes in $\bbz_6$ of 1, 2, and 3.

According to Chapman and Smith~\cite[Theorem~3.8]{MR1074505}, the ring $D$ is an HFD.
Thus, we need only show that $D$ is not an IDPD.
To this end, suppose by way of contradiction that $D$ were an IDPD.

Let $\p_1$, $\p_2$, $\q$, $\fr_1$, and $\fr_2$
be height-1 prime ideals of $D$ such that $1\approx \p_i$ 
and $2\approx \q$ and $3\approx \fr_i$ for $i=1,2$
in the notation of~\ref{notn170906a}.
In particular, we have $\p_2\equiv\p_1$ and $\q\equiv\p_1^{(2)}$ and $\fr_i\equiv\p_1^{(3)}$ for $i=1,2$.
Consider the ideal 
$$\fa=\p_1^{(3)}\cap\p_2^{(2)}\cap\q^{(2)}\cap\fr_1^{(2)}\cap\fr_2.$$
In $\Cl(D)\cong\bbz_6$, our choice of $S$ implies that
$$\fa\equiv\p_1^{(5+2\cdot 2+3\cdot 3)}=\p_1^{(18)}\equiv D.$$
That is, $\fa$ is principal, say $\fa=zD$. 

Next, consider the ideal
$$\fb=\p_1\cap\p_2^{(2)}\cap\fr_1\equiv\p_1^{(6)}\equiv D.$$
It follows that $\fb$ is also principal, say, $\fb=\pi D$.
Each $\p_i$ has order $o(\p_i)=6$ in $\Cl(D)$, and similarly $o(\q)=3$ and $o(\fr_i)=2$.
So, in the notation of~\ref{notn170805a}, we have $\ell(\pi)=1$, and Remark~\ref{disc170805a}
implies that $\pi$ is an atom in $D$.
Furthermore, comparing the exponents on the primes defining $\fa$ and $\fb$, 
it is straightforward to show that $zD=\fa\subseteq\fb=\pi D$, so we have $\pi\mid z$.

Similarly, the ideal
$$\fc=\p_1^{(3)}\cap\fr_2$$
is principal, generated by an atom $\tau$ such that $\tau\mid z$.
The fact that $\fb$ and $\fc$ use different exponents on their primes (or since they don't use all the same primes)
implies that $\fb\neq\fc$, so $\pi\not\sim\tau$.

Since $D$ is an IDPD, it follows that there are atoms $\pi',\tau'\in D$ such that $\pi\pi'\sim\tau\tau'\mid z$.
In particular, we have $\fa=zD\subseteq\pi'D$ so the defining expression of $\fa$ implies that 
we can write
$$\pi'D
=\p_1^{(a_1)}\cap\p_2^{(a_2)}\cap\q^{(b)}\cap\fr_1^{(c_1)}\cap\fr_2^{(c_2)}$$
where the exponents are in $\bbn$.
Moreover the condition $\pi\pi'\mid z$ implies that
$$\p_1^{(3)}\cap\p_2^{(2)}\cap\q^{(2)}\cap\fr_1^{(2)}\cap\fr_2=\fa\subseteq\pi\pi'D
=\p_1^{(1+a_1)}\cap\p_2^{(2+a_2)}\cap\q^{(b)}\cap\fr_1^{(1+c_1)}\cap\fr_2^{(c_2)}.$$
Comparing exponents, we conclude that
\begin{align*}
1+a_1&\leq 3
&
2+a_2&\leq 2
&
b&\leq 2
&
1+c_1&\leq 2
&
c_2&\leq 1
\intertext{so we must have}
a_1&\leq 2
&
a_2&\leq 0
&
b&\leq 2
&
c_1&\leq 1
&
c_2&\leq 1
\end{align*}
Similarly, we have
$$\tau'D
=\p_1^{(d_1)}\cap\p_2^{(d_2)}\cap\q^{(e)}\cap\fr_1^{(f_1)}\cap\fr_2^{(f_2)}$$
and
$$\p_1^{(3)}\cap\p_2^{(2)}\cap\q^{(2)}\cap\fr_1^{(2)}\cap\fr_2=\fa\subseteq\tau\tau'D
=\p_1^{(3+d_1)}\cap\p_2^{(d_2)}\cap\q^{(e)}\cap\fr_1^{(f_1)}\cap\fr_2^{(1+f_2)}.$$
Comparing exponents, we conclude that
\begin{align*}
3+d_1&\leq 3
&
d_2&\leq 2
&
e&\leq 2
&
f_1&\leq 2
&
1+f_2&\leq 1
\intertext{so we must have}
d_1&\leq 0
&
d_2&\leq 2
&
e&\leq 2
&
f_1&\leq 2
&
f_2&\leq 0.
\end{align*}
Furthermore, the condition $\pi\pi'\sim\tau\tau'$ implies that $\pi\pi'D=\tau\tau'D$.
Comparing the exponents for these ideals in the displays above, we conclude that
\begin{align*}
1+a_1&=3+d_1=3
&
d_2&=2+a_2=2
&
b&=e\geq 0
\\
f_1&=1+c_1\geq 1
&
c_2&=1+f_2=1.
\end{align*}
From this, we conclude that 
$$\pi\pi'D=\tau\tau'D
=\p_1^{(3)}\cap\p_2^{(2)}\cap\q^{(e)}\cap\fr_1^{(f_1)}\cap\fr_2^{(1)}$$
where $0\leq e\leq 2$ and $1\leq f_1$. Write $f_1=1+x$ where $x\geq 0$.

Since $D$ is an HFD, we must have
$$
2
=\ell(\pi\pi')
=\frac 56+\frac e3+\frac{f_1+1}2
=\frac 56+\frac e3+\frac{x+2}2
=\frac {11}6+\frac e3+\frac x2.$$
It follows that $e$ and $x$ are non-negative integers such that $\frac e3+\frac x2=\frac 16$, which is impossible.
This contradiction completes the proof.
\end{proof}

\begin{disc}\label{disc170908a}
It is worth noting that the ring $D$ from Theorem~\ref{thm170908a} has the smallest divisor class group among all
Krull domains that are HFDs but not IDPDs.
Indeed, if $R$ is a Krull domain with $|\Cl(R)|<6$ that is also an HFD, then $\Cl(R)$ is isomorphic to one of the following:
$0$, $\bbz_2$, $\bbz_3$, $\bbz_4$, $\bbz_2^2$, or $\bbz_5$.
In the cyclic cases, Theorem~\ref{thm170903b} shows that $R$ is an IDPD.
In the case of $\bbz_2^2$, we get the same conclusion from Corollary~\ref{thm170802a}.

Furthermore, if $D'$ is another Krull domain with $\Cl(D')\cong \bbz_6$ such that $S'=S(D')\neq\{1,2,3\}$, then $D'$ is an IDPD.
(In other words, the $S$ of Theorem~\ref{thm170908a} is the only way to find a 
Krull domain $D$ with $\Cl(D)\cong \bbz_6$ that is not an IDPD.)
We must be a bit careful here since, technically, we have only defined $S'$ when $D'$
has a height-1 prime ideal representing an element of order 6 in $\Cl(D')$.
In this case, let $\phi\colon\Cl(D)\to\bbz_6$ be some isomorphism, and set 
$$S'=\{j\in\{1,2,3,4,5\}\mid\text{$\phi^{-1}(j)$ is the class of a height-1 prime}\}.
$$
We now analyze $D'$ by cases.

Case 1: $1\in S'$. In this case, Lemma~\ref{lem170906c} implies that $S'\subseteq\{1,2,3\}$.
We have assumed that $S'\neq\{1,2,3\}$, so we have three subcases.
If $S'=\{1\}$, then Lemma~\ref{lem170903a} shows that $D'$ is an IDPD.
On the other hand, if
$S'=\{1,2\}$
or $S'=\{1,3\}$, then $D'$ is an IDPD by Theorem~\ref{thm170906a}.

Case 2: $1\notin S'$. In this case, by the specific construction of $S'$ in Notation~\ref{notn170906a},
the set $S'$ contains no integer representing an element of order $6$ in $\bbz_6$. In other words, in this case, we
must have $S'\subseteq\{2,3,4\}$. 
Since the classes of prime ideals in $\Cl(D')$ generate $\Cl(D')$, it follows that we must have $2,3\in S'$ or $3,4\in S'$.
On the other hand, Lemma~\ref{prop170909a} shows that $S'$ cannot contain both 2 and 4.
Thus, we must have $S'=\{2,3\}$ or $S'=\{3,4\}$. In either case one obtains an isomorphism $\Cl(D')\cong\bbz_2\times\bbz_3$
satisfying the hypotheses of Lemma~\ref{lem170903a}, so we conclude that $D'$ is an IDPD, as desired.
\end{disc}

\begin{disc}\label{disc170909a}
One can construct other Dedekind domains with the properties of the 
ring $D$ from Theorem~\ref{thm170908a}. 
For example, let $E$ be a Dedekind domain with $\Cl(E)\cong\bbz_4\times\bbz_2$
such that the height-1 primes correspond to the elements $(1,0)$, $(0,1)$, and $(3,1)$.
Let $\p$, $\q$, and $\fr$ be such primes, respectively.
Then the ideals 
$zD=\p^{(4)}\cap\q^{(2)}\cap\fr^{(4)}$
and
$\pi D=\p\cap\q\cap\fr$
and
$\tau D=\fr^{(4)}$
show that $E$ cannot be an IDPD, as in the proof of 
Theorem~\ref{thm170908a}. 
\end{disc}

The form of the divisor class groups in Theorem~\ref{thm170908a} and Remark~\ref{disc170909a} lead us to the following.

\begin{question}\label{q170911a}
For any $m\in\bbn$ with $m\geq 3$, does there exist a Dedekind domain $D$ that is an HFD but not an IDPD
and has $\Cl(D)\cong\bbz_m\times\bbz_2$?
More generally, for any
$m,n\in\bbn$ with $m\geq n\geq 2$, does there exist a Dedekind domain $D$ that is an HFD but not an IDPD
and has $\Cl(D)\cong\bbz_m\times\bbz_n$?
\end{question}

\appendix
\section{Ancillary Lemmas}\label{sec170903a}

This appendix contains three technical lemmas for use in the proofs of Section~\ref{sec170906a}.
We would not be surprised to learn that these lemmas are known and that more general statements are known, 
but we include them for the sake of completeness.

\begin{lem}\label{lem170903b}
Let $T,n\in \bbn$ be such that $n\geq 2$, and let
$M_1,\ldots,M_T,c_1,\ldots,c_T\in\bbn_0$ be such that
$M_t\leq c_t$ for $t=1,\ldots,T$ and such that $\sum_{t=1}^TM_t\leq 2n\leq\sum_{t=1}^tc_t$.
Then there are integers $y_1,\ldots,y_T\in\bbn_0$ such that
$y_t\leq c_t-M_t$ for all $t$ and such that $\sum_{t=1}^TM_t+\sum_{t=1}^Ty_t=2n$.
\end{lem}

\begin{proof}
We argue by induction on $m=2n-\sum_{t=1}^TM_t\geq 0$.
The base case $m=0$ is trivial with $y_t=0$ for all $t$. 

In the induction step with $m\geq 1$, note that we must have $M_t< c_t$ for some $t$;
otherwise, we have $M_t=c_t$ for all $t$ so $m=0$.
Re-order the $M_t$'s and $c_t$'s to assume that $M_T<c_T$.
Let $M'_t=M_t$ for all $t<T$, and set $M'_T=M_T+1$.
Then we have $M'_t\leq c_t$ for all $t$ and
$m'=2n-\sum_{t=1}^TM'_t=(2n-\sum_{t=1}^TM_t)-1\geq 0$.
Thus, by our (unstated) induction hypothesis, there are integers 
$y'_1,\ldots,y'_T\in\bbn_0$ such that
$y'_t\leq c_t-M'_t$ for all $t$ and such that $\sum_{t=1}^TM'_t+\sum_{t=1}^Ty'_t=2n$.
Set $y_t=y'_t\geq 0$ for all $t<T$ and set $y_T=y'_T+1\geq 0$.

For $t<T$, note that $y_t=y'_t \leq c_t-M'_t=c_t-M_t$.
Also, we have $y_T=y'_T+1\leq c_T-M'_T+1=c_T-(M_T+1)+1=c_T-M_T$.
Similarly, we have
$\sum_{t=1}^TM_t+\sum_{t=1}^Ty_t=\sum_{t=1}^TM'_t+\sum_{t=1}^Ty'_t=2n$,
so the numbers $y_t$ satisfy the desired condition.
\end{proof}

\begin{lem}\label{lem170906a}
Let $r,n\in\bbn_0$, and let $S$ be a non-empty set of positive divisors of $n$. 
Assume that $S$ is totally ordered by divisibility. 
Let $Y$ be a non-empty finite set 
partitioned into a union $Y=\bigcup_{s\in S}Y_s$ of pairwise disjoint subsets.
Let functions $M,a\colon Y\to\bbn_0$ be given such that $M(y)\leq a(y)$ for all $y\in Y$.
Assume that $\sum_{s\in S}s\sum_{y\in Y_s}M(y)\leq 2n\leq rn=\sum_{s\in S}s\sum_{y\in Y_s}a(y)$.
Then there is a function $f\colon Y\to\bbn_0$ such that $M(y)+f(y)\leq a(y)$ for all $y\in Y$
and such that $\sum_{s\in S}s\sum_{y\in Y_s}(M(y)+f(y))= 2n$.
\end{lem}

\begin{proof}
We argue by induction on
$m=2n-\sum_{s\in S}s\sum_{y\in Y_s}M(y)$,
as in the proof of Lemma~\ref{lem170903b}.
The base case $n=0$ is trivial with $f(y)=0$ for all $y\in Y$.

For the induction step, assume that $m\geq 1$ and that the result holds for all functions $M'$
with $2n-\sum_{s\in S}s\sum_{y\in Y_s}M'(y)<m$.
The condition $m\geq 1$ implies that 
$\sum_{s\in S}s\sum_{y\in Y_s}M(y)<\sum_{s\in S}s\sum_{y\in Y_s}a(y)$
so there is some $s\in S$ and some $y\in Y_s$ such that $M(y)<a(y)$. 
Set 
$$s_0=\min\{s\in S\mid\text{there is a $y\in Y_s$ such that $M(y)<a(y)$}\}$$
and fix an element $y_0\in Y_{s_0}$ such that $M(y_0)<a(y_0)$.
Define $M'\colon Y\to\bbn_0$ by the formula
$$M'(y)=
\begin{cases}
M(y)&\text{if $y\neq y_0$} \\
M(y)+1&\text{if $y= y_0$}
\end{cases}
$$ 
and note that $M'(y)\leq a(y)$ for all $y\in Y$, by construction.

Claim 1:
We have $r_0\mid m$. 
We work modulo $s_0$:
\begin{align*}
m
&\stackrel{(1)}=2n-\sum_{s\in S}s\sum_{y\in Y_s}M(y)
\\
&\stackrel{(2)}=2n-\sum_{\substack{s\in S\\ s\geq s_0}}s\sum_{y\in Y_s}M(y)-\sum_{\substack{s\in S\\ s< s_0}}s\sum_{y\in Y_s}M(y)
\\
&\stackrel{(3)}\equiv -\sum_{\substack{s\in S\\ s< s_0}}s\sum_{y\in Y_s}M(y)\pmod{s_0}
\\
&\stackrel{(4)}\equiv -\sum_{\substack{s\in S\\ s< s_0}}s\sum_{y\in Y_s}a(y)
\\
&\stackrel{(5)}=-\sum_{\substack{s\in S\\ s\geq s_0}}s\sum_{y\in Y_s}a(y)-\sum_{\substack{s\in S\\ s< s_0}}s\sum_{y\in Y_s}a(y)\pmod{s_0}
\\
&\stackrel{(6)}=-\sum_{s\in S}s\sum_{y\in Y_s}a(y)
\\
&\stackrel{(7)}\equiv 0\pmod{s_0}
\end{align*}
Step~(1) is by definition of $m$.
Steps~(2) and~(6) are simple grouping.
Steps~(3) and~(5) are from the divisibility relations $s_0\mid n$ and $s_0\mid s$ for all $s\in S$ with $s\geq s_0$.
Step~(4) is from the definition of $s_0$. 
Step~(7) is from the assumption $rn=\sum_{s\in S}s\sum_{y\in Y_s}a(y)$.
This establishes  Claim~1.

Claim 2: $m\geq s_0$.
By way of contradiction, suppose that $m<s_0$.
By assumption, we have $0< m$.
However, Claim~1 shows that $s_0\mid m$, which is impossible since $0< m< s_0$.
This establishes  Claim~2.

To complete the proof, set 
$m'=2n-\sum_{s\in S}s\sum_{y\in Y_s}M'(y)$.
As we have already noted, we have
$M'(y)\leq a(y)$ for all $y\in Y$.
Now, we compute.
\begin{align*}
m'
&\stackrel{(8)}=2n-\sum_{s\in S}s\sum_{y\in Y_s}M'(y)
\\
&\stackrel{(9)}=2n-\sum_{s\in S}s\sum_{y\in Y_s}M(y)-s_0
\\
&\stackrel{(10)}=m-s_0
\\
&\stackrel{(11)}\geq 0
\end{align*}
Step~(8) is by definition of $m'$.
Step~(9) is by definition of $M'$.
Step~(10) is by definition of $m$.
And 
Step~(11) is from Claim~2.

It follows that the function $M'$ satisfies our induction hypothesis.
Thus, there is a function $f'\colon Y\to\bbn_0$ such that $M'(y)+f'(y)\leq a(y)$ for all $y\in Y$
and such that $\sum_{s\in S}s\sum_{y\in Y_s}(M'(y)+f'(y))= 2n$.
Define $f\colon Y\to\bbn_0$ by the formula
$$f(y)=
\begin{cases}
f'(y)&\text{if $y\neq y_0$} \\
f'(y)+1&\text{if $y= y_0$.}
\end{cases}
$$ 
By construction, we have $M(y)+f(y)=M'(y)+f'(y)$ for all $y\in Y$.
We conclude that $f$ satisfies the desired conditions.
\end{proof}

\begin{lem}\label{lem170903c}
Let $I,p,r\in\bbn_0$ be such that $p,r\geq 2$. 
Let $X$ be a non-empty finite set 
partitioned into a union of pairwise disjoint subsets $X_0,\ldots,X_I$.
Let functions $M,a\colon X\to\bbn_0$ be given such that $M(x)\leq a(x)$ for all $x\in X$.
Assume that $\sum_{i=0}^Ip^{I-i}\sum_{x\in X_i}M(x)\leq 2p^I\leq rp^I=\sum_{i=0}^Ip^{I-i}\sum_{x\in X_i}a(x)$.
Then there is a function $f\colon X\to\bbn_0$ such that $M(x)+f(x)\leq a(x)$ for all $x\in X$
and such that $\sum_{i=0}^Ip^{I-i}\sum_{x\in X_i}(M(x)+f(x))= 2p^I$.
\end{lem}

\begin{proof}
Use  $Y=X$ and $n=p^I$ in Lemma~\ref{lem170906a} with
$S=\{p^i\mid i=0,\ldots,I\}$ and $X_i=Y_{p^{I-i}}$.
\end{proof}

\section*{Acknowledgments}
I am grateful to Scott Chapman, Jim Coykendall, and Tiberiu Dumitrescu for useful suggestions about this work.

\providecommand{\bysame}{\leavevmode\hbox to3em{\hrulefill}\thinspace}
\providecommand{\MR}{\relax\ifhmode\unskip\space\fi MR }
\providecommand{\MRhref}[2]{%
  \href{http://www.ams.org/mathscinet-getitem?mr=#1}{#2}
}
\providecommand{\href}[2]{#2}

\end{document}